\documentclass[reqno, 12pt]{amsart}%
\usepackage{amsmath}
\usepackage{amssymb}
\usepackage{amsfonts}
\usepackage{graphicx}%
\providecommand{\U}[1]{\protect\rule{.1in}{.1in}}
\newtheorem{theorem}{Theorem}
\theoremstyle{plain}

\newtheorem{corollary}{Corollary}

\newtheorem{definition}{Definition}

\newtheorem{lemma}{Lemma}

\newtheorem{proposition}{Proposition}
\newtheorem{remark}{Remark}

\DeclareMathOperator{\Div}{div}
\numberwithin{equation}{section}
\numberwithin{theorem}{section}
\numberwithin{proposition}{section}
\numberwithin{remark}{section}
\numberwithin{definition}{section}
\numberwithin{lemma}{section}
\numberwithin{corollary}{section}
\numberwithin{example}{section}
\numberwithin{claim}{section}
\begin{document}
\title[Nonlinear Parabolic Equation with Hysteresis]{Well-posedness and Long-time Behaviour for a Nonlinear Parabolic Equation with Hysteresis}
\author{Achille Landri Pokam Kakeu}
\address{A.L. Pokam Kakeu, Department of Mathematics and Computer Science, University
of Dschang, P.O. Box 67, Dschang, Cameroon}
\email{pokakeu@yahoo.fr}
\author{Jean Louis Woukeng}
\address{J.L. Woukeng, Department of Mathematics and Computer Science, University of
Dschang, P.O. Box 67, Dschang, Cameroon}
\email{jwoukeng@yahoo.fr}
\date{August, 2019}
\subjclass[2000]{47J10, 74N30}
\keywords{Nonlinear parabolic equation; hysteresis; time discretization method,
long-time behaviour.}

\begin{abstract}
The work deals with a study of a nonlinear parabolic equation with hysteresis,
containing a nonlinear monotone operator in the diffusion term. The
well-posedness of the model equation is addressed by using an implicit time
discretization scheme in conjunction with the piecewise monotonicity of the
hysteresis operator, and a fundamental inequality due to M. Hilpert. A
characterization of the $\omega$-limit set of the solution is then given
through the study of the long-time behaviour of the solution of the equation
in which we investigate the convergence of trajectories to limit points.

\end{abstract}
\maketitle

\section{Statement of the problem\label{subsec1.1}}

We consider a nonlinear parabolic problem with hysteresis functionals whose
the diffusion term is a monotone operator arising from a convex functional.
The main purpose is to study the well-posedness and the long-time behaviour.
The model problem is stated as follows %

\begin{equation}
\left\{
\begin{array}
[c]{l}%
\dfrac{\partial}{\partial t}\left(  cu+w\right)  -{\Div}\boldsymbol{a}%
(\cdot,\nabla u)=f\text{ in }Q=\Omega\times(0,T)\\
\\
w(x,t)=\mathcal{W}[u(x,\cdot);x](t)\text{ in }Q\\
\\
u=0\text{ on }\partial\Omega\times(0,T)\text{ and }u(x,0)=u_{0}(x)\text{ in
}\Omega
\end{array}
\right.  \ \ \ \ \label{eq1}%
\end{equation}
where $u$ is the unknown function, $T>0$ is the final time, and $\Omega$ is a
sufficiently smooth open bounded set in $\mathbb{R}^{d}$ locally located on
one side of its boundary. The known data in (\ref{eq1}) are the functions $c$,
$f$, $\mathcal{W}$, $\boldsymbol{a}$ and $u_{0}$, which are constrained as follows.

\begin{itemize}
\item[(\textbf{A1})] The function $\boldsymbol{a}=(a_{i})_{1\leq i\leq d}$ is
defined by $a_{i}(x,\lambda)=\frac{\partial J}{\partial\lambda_{i}}%
(x,\lambda)$ where the function $J:\Omega\times\mathbb{R}^{d}\rightarrow
\lbrack0,+\infty)$ satisfies the following conditions:

\item[(i)] $J(\cdot,\lambda)$ is measurable and differentiable for all $\lambda\in\mathbb{R}^{d}$,

\item[(ii)] $J(x,\cdot)$ is strictly convex for almost all $x\in\Omega$,

\item[(iii)] There exist three constants $p\geq2$, $\alpha_{1}>0$ and
$\alpha_{2}>0$ such that
\begin{equation}
\alpha_{1}\left\vert \lambda\right\vert ^{p}\leq J(x,\lambda)\leq\alpha
_{2}(1+\left\vert \lambda\right\vert ^{p})\;\;\;\;\;\;\;\;\;\;\;\;\;
\label{0.2}%
\end{equation}
for all $\lambda\in\mathbb{R}^{d}$ and for almost all $x\in\Omega$.

\item[(\textbf{A2})] $c\in L^{\infty}(\Omega)$ with $c\geq\alpha>0$ where
$\alpha$ is a constant independent of $x\in\Omega$.

\item[(\textbf{A3})] $f\in L^{2}(Q)$ and $u_{0}\in W_{0}^{1,p}(\Omega)$.

\item[(\textbf{A4})] For every $x\in\Omega$, the hysteresis operator
$\mathcal{W}[\cdot;x]$ is continuous on $\mathcal{C}\left(  [0,T]\right)  $
and piecewise increasing. Moreover $\mathcal{W}$ is affine bounded and there
exist a function $\kappa_{0}\in L^{2}(\Omega)$ and a positive constant
$\gamma_{0}$ such that, for all $\ell\in\mathbb{N}$, the parameterized final
value mapping
\[
(s,x)\mapsto\mathcal{W}_{f}(s;x),\ s=(v_{0},...,v_{\ell})\in S
\]
is measurable and satisfies
\begin{equation}
\left\vert \mathcal{W}_{f}(s;x)\right\vert \leq\kappa_{0}(x)+\gamma
_{0}\left\Vert s\right\Vert _{\infty}, \label{eq1.11}%
\end{equation}
where $\mathcal{W}_{f}$ denotes the generating functional of the hysteresis
operator $\mathcal{W}$ and $S$ is the set of all finite strings of real
numbers, a string being as usual a vector having either finitely or countably
infinitely many real components.
\end{itemize}

\noindent Further information concerning the construction of $\mathcal{W}_{f}$
can be found in  \cite{Brokate}.

Besides piecewise monotonicity and continuity, we need a further assumption on
the hysteresis operator $\mathcal{W}[\cdot;x]$, which will ensure uniqueness
of the solution:

\begin{itemize}
\item[(\textbf{A5})] The operator $\mathcal{W}[\cdot;x]$ maps $W^{1,1}(0,T)$
for every $x\in\Omega$ into itself, and there exist $\gamma_{1}>0$ and
$\kappa_{1}\in L^{2}(\Omega)$ such that the condition
\[
\left\vert (\mathcal{W}[v;x])^{\prime}(t)\right\vert \leq\kappa_{1}%
(x)+\gamma_{1}\left\vert v^{\prime}(t)\right\vert \ \ \forall x\in
\Omega,\ \text{for a.e. }t\in(0,T)
\]
is satisfied for every $v\in W^{1,1}(0,T)$.
\end{itemize}

Of special interest as far as applications are concerned in this work is an
existence and uniqueness result and an asymptotic behaviour in time for an
evolution parabolic equation modeling a diffusion process with hysteresis;
these questions are addressed in this paper. The problem (\ref{eq1}) is an
equation that can be regarded as a model of heat conduction including phase
transition.

Hysteresis is defined as rate independent memory effect. The basic feature of
hysteresis behaviour is a memory effect and the irreversibility of the
process. Hysteresis was mentioned for the first time in an article
\cite{Ewing} on magnetism published in 1885. It is a nonlinear phenomenon that
occurs in many natural and constructed systems, and because of the strong
nonlinearity of this phenomenon which is usually non-smooth, it has not been
easy to treat it mathematically for a very long time. Hence it was only in the
early seventies that a group of Russian scientists led by M. A. Krasnoselskii 
initiated a systematic mathematical investigation of the phenomenon of hysteresis which resulted in
the fundamental monograph of Krasnoselskii and Pokrovskii \cite{Pokrovskii}.
During that time, many mathematicians have contributed to the mathematical
theory, and the important monographs of Mayergoyz \cite{Mayergoyz} and
Visintin \cite{A.Visintin} have emerged. It is very important to note that
Visintin intensively investigated PDEs with hysteresis. 

Hysteresis operators can be seen as nonlinear
causal functional operators. One of the main characteristics of these
operators is the fact that they have memory character, \textit{i.e.} the value
at some time $t$ do not only depends on the value of $t$ at this precise
moment, but it also depends on the previous evolutions and on inputs up to the
time $t$. For further results
and references concerning hysteresis operators, see e.g. \cite{Eleuteri,
Eleuteri2, Flynn, Francu, Visintin} and the references therein.

The work is organized as follows. In Section \ref{subsec1.3}, we gather
necessary elementary tools together with some functional spaces, and we state
our main result. Suitable properties of the function $\boldsymbol{a}%
(x,\lambda)$ are detailed in Section \ref{subsec1.2}.  Section \ref{sec3} is
dedicated to the proof of the existence result. The technique we use for
proving the existence result is based on approximation by implicit time
discretization, a priori estimates and passage to the limit by compactness.
This approximation procedure is often used and is quite convenient in the
analysis of equations that include a memory operator, as in any time-step we
solve a stationary problem in which this operator is reduced to a
superposition with a nonlinear function. The details of the passage to the
limit in the nonlinear diffusion term is worked out carefully. To obtain
uniqueness result, a fundamental inequality due to Hilpert \cite{Hilpert} is
employed as well as the $L^{1}$-stability of the solution of the equation. The
latter is done in Section \ref{subsec3.2}. Finally, we prove in the last
section a result related to the long-time behaviour of the solution to
(\ref{eq1}).

\section{Functional setting and statement of the main results
\label{subsec1.3}}

\subsection{Functional setting}

In order to facilitate the reading of this paper, we collect here some
mathematical tools, starting with some well-known inequalities which will be
used in the work.

The set of non-negative integers is denoted by $\mathbb{N}$ and
 $\mathbb{N}^{\ast}\equiv \mathbb{N} - \left\{0\right\}$
while $\mathbb{R}$ stands for the set of real numbers. As usual we denote the
non-negative real numbers by $\mathbb{R}_{+}$. For $1\leq d\in\mathbb{N}$,
$\mathbb{R}^{d}$ stands for the numerical space of variables $x=(x_{1}%
,...,x_{d})$. For any Banach space $X$, we shall denote by $X^{\prime}$ its
topological dual.

We will need some fundamental inequalities which are given below:

\begin{lemma}
[Young inequality]Suppose that $1<p,p^{\prime}<\infty$ and $\frac{1}{p}%
+\frac{1}{p^{\prime}}=1$ then
\[
\left\vert ab\right\vert \leq\frac{1}{p}\delta\left\vert a\right\vert
^{p}+\frac{1}{p^{\prime}}\delta^{-\frac{1}{p-1}}\left\vert b\right\vert
^{p^{\prime}}\;\;\forall a,b\in\mathbb{R},\;\delta>0.
\]

\end{lemma}

\begin{lemma}
[Poincar\'{e} inequality]Let $\Omega$ be an open bounded set in $\mathbb{R}%
^{d}$ and $p\geq1$ a real number. Then there is a constant $C=C(\Omega,p)>0$
such that
\[
\left\Vert u\right\Vert _{L^{p}(\Omega)}\leq C\left\Vert \nabla u\right\Vert
_{L^{p}(\Omega)}\text{ for all }u\in W_{0}^{1,p}(\Omega).
\]

\end{lemma}

For $m\in\mathbb{N}$ and $1\leq p\leq\infty$, the Sobolev space $W^{m,p}%
(\Omega)$ is defined as the space of all functions $u\in L^{p}(\Omega)$ having
generalized partial derivatives $D^{\alpha}u\in L^{p}(\Omega)$ for every
multi-index $\alpha$ satisfying $0\leq\left\vert \alpha\right\vert \leq m$.
Endowed with the norm
\[
\left\Vert u\right\Vert _{W^{m,p}(\Omega)}=\left(  \sum_{0\leq\left\vert
\alpha\right\vert \leq m}\int_{\Omega}{\left\vert D^{\alpha}u(x)\right\vert
^{p}dx}\right)  ^{\frac{1}{p}},
\]
$W^{m,p}(\Omega)$ becomes a Banach space which is separable for $1\leq
p<\infty$. In the case $p=2$, one obtains a Hilbert space denoted by
$H^{m}(\Omega)$ and with the inner product
\[
\left\langle u,v\right\rangle =\sum_{0\leq\left\vert \alpha\right\vert \leq
m}\int_{\Omega}{D^{\alpha}u(x)D^{\alpha}v(x)dx}.
\]
Also, denoting by $\mathcal{C}_{0}^{\infty}(\Omega)$ the space of infinitely
differentiable functions on $\Omega$ with compact supports, we define the
space $W_{0}^{m,p}(\Omega)$ as the closure of $\mathcal{C}_{0}^{\infty}%
(\Omega)$ in $W^{m,p}(\Omega)$, and we denote by $W^{-m,p^{\prime}}(\Omega)$
the topological dual of $W_{0}^{m,p}(\Omega)$ (integers $m\geq1$ and real
number $p>1$). For $p=2$, instead of $H^{1}(\Omega)$, we may in general define
the fractional Sobolev spaces $H^{\sigma}(\Omega)$ with $\sigma\in\mathbb{R}$
as in \cite{Adams, Brezis}.

Next, we need to introduce another class of function spaces which will be
employed for the variational treatment of our evolution problem, namely spaces
of the type $W^{m,p}(0,T;V)$, where $T>0$ is some final time and $V$ is a
certain function space. If $V=\mathbb{R}^{d}$, no additional difficulties
occur, since all relevant properties of $W^{m,p}(0,T)$ carry over to the
finite product $\prod_{i=1}^{d}{W^{m,p}(0,T)}$. In the infinite-dimensional
case, the definition of $W^{m,p}(0,T;V)$ uses the notion of Bochner integrals
which attain values in $V$. Let us give a very brief introduction to this
notion. For our purposes, we may restrict ourselves to the case where $V$ is a
reflexive Banach space, since we will exclusively deal with the space
$V=W_{0}^{1,p}(\Omega)$.

\begin{definition}
\emph{A function }$u:\left[  0,T\right]  \rightarrow V$\emph{ is called
Bochner measurable, if it is the pointwise limit of a sequence }$(u_{n}%
)$\emph{ of simple functions. The function }$u$\emph{ is called Bochner
integrable if }%
\[
\lim_{n\rightarrow\infty}{\int_{0}^{T}{\left\Vert u(t)-u_{n}(t)\right\Vert
_{V}}dt}=0,
\]
\emph{in which case the integral of }$u$\emph{ is defined by }%
\[
\int_{0}^{T}{u(t)dt}=\lim_{n\rightarrow\infty}\int_{0}^{T}{u_{n}(t)dt}.
\]
\emph{We denote by }$L^{p}(0,T;V)$\emph{ (}$1\leq p<\infty$\emph{) the space
of all Bochner measurable functions }$u:\left[  0,T\right]  \rightarrow
V$\emph{ for which }%
\[
\left\Vert u\right\Vert _{L^{p}(0,T;V)}=\left(  \int_{0}^{T}{\left\Vert
u(t)\right\Vert _{V}^{p}dt}\right)  ^{\frac{1}{p}}<\infty.
\]

\end{definition}

\medskip Equipped with $\left\Vert \cdot\right\Vert _{L^{p}(0,T;V)}$,
$L^{p}(0,T;V)$ is a Banach space. Similarly, we define the space $L^{\infty
}(0,T;V)$, using the norm
\[
\left\Vert u\right\Vert _{L^{\infty}(0,T;V)}=\text{ess}\sup_{t\in\left[
0,T\right]  }\left\Vert u(t)\right\Vert _{V}.
\]
For $1<p<\infty$, the space $L^{p}(0,T;V)$ is separable. In addition, its
topological dual is isomorphic to $L^{p^{\prime}}(0,T;V^{\prime})$, where
$\frac{1}{p}+\frac{1}{p^{\prime}}=1$. After the definition of the spaces
$L^{p}(0,T;V)$, the spaces $W^{m,p}(0,T;V)$ are introduced using the concept
of distributions with values in Banach spaces. For the details of this
construction, we refer the reader to \cite{Adams}.

We end this subsection with an important result related to the existence
result for monotone operators. Let $X$ be a real reflexive Banach space, let
$A:X\rightarrow X^{\prime}$ ($X^{\prime}$ the topological dual of $X$) and let
$\left\langle ,\right\rangle $ denote the duality pairing between $X$ and
$X^{\prime}$. We recall the following definitions. The operator

\begin{itemize}
\item $A$ is monotone if $\left\langle Au_{1}-Au_{2},u_{1}-u_{2}\right\rangle
\geq0$ for all $u_{1},u_{2}\in X$;

\item $A$ is strictly monotone if $\left\langle Au_{1}-Au_{2},u_{1}%
-u_{2}\right\rangle >0$ whenever $u_{1}\neq u_{2}$;

\item $A$ is hemicontinuous if $\lim_{t\rightarrow0}A(u+tv)=Au$ in $X^{\prime
}$-weak$\ast$ for all $u,v\in X$.

\item $A$ is coercive if
\[
\lim_{\left\Vert u\right\Vert \rightarrow\infty}\frac{\left\langle
Au,u\right\rangle }{\left\Vert u\right\Vert }=\infty.
\]

\end{itemize}

Let us consider the operator equation of the form
\begin{equation}
\text{Find }u\in X\text{ such that }Au=b\text{.} \label{eq1000}%
\end{equation}
The existence issue for (\ref{eq1000}) is given by the next result.

\begin{theorem}
[Browder-Minty]\label{t2.4}Suppose that $A$ is strictly monotone,
hemicontinuous and coercive. Then \emph{(\ref{eq1000})} has a unique solution
$u\in X$ for every $b\in X^{\prime}$.
\end{theorem}

The proof of the above theorem can be found in \cite{Lions} or alternatively
in \cite{Lazlo} where several applications have been given for various
properties of monotone operators.

\subsection{Statement of the main results\label{subsec1.4}}

We first define the notion of weak solution we will deal with in this work.

\begin{definition}
\label{d1.1}\emph{Let the assumptions (\textbf{A1})-(\textbf{A4}) hold. We say
that a function }$u:Q\rightarrow\mathbb{R}$\emph{ is a weak solution of
(\ref{eq1}) if }%
\[
\left\{
\begin{array}
[c]{l}%
u\in L^{\infty}(0,T;W_{0}^{1,p}(\Omega))\text{\emph{ with }}u^{\prime}\in
L^{2}(0,T;L^{2}(\Omega)),\\
w=\mathcal{W}(u;\cdot)\in L^{2}(Q)\cap L^{2}(\Omega;\mathcal{C}%
([0,T]))\text{\emph{ with }}w^{\prime}\in L^{p^{\prime}}(0,T;W^{-1,p^{\prime}%
}(\Omega))
\end{array}
\right.
\]
\emph{and }$u$\emph{ satisfies equation (\ref{eq14}) }%
\begin{equation}
\left\{
\begin{array}
[c]{l}%
{%
{\displaystyle\int_{Q}}
{cu^{\prime}(x,t)\varphi(x,t)dx}dt}+%
{\displaystyle\int_{0}^{T}}
\left\langle {w^{\prime}(\cdot,t),\varphi(\cdot,t)}\right\rangle {dt}\\
\ \ \ +{%
{\displaystyle\int_{Q}}
{\boldsymbol{a}(x,\nabla u(x,t))\cdot\nabla\varphi(x,t)dx}dt}={%
{\displaystyle\int_{Q}}
{f(x,t)\varphi(x,t)dx}dt}\\
\text{\emph{for all }}\varphi\in L^{p}(0,T;W_{0}^{1,p}(\Omega)).
\end{array}
\right.  \label{eq14}%
\end{equation}

\end{definition}

The first main purpose of the work is to prove the following result.

\begin{theorem}
\label{t4.1}Let the assumptions \emph{(\textbf{A1})-(\textbf{A4})} hold. Then
there exists at least a solution $u$ in the sense of Definition
\emph{\ref{d1.1}}. Moreover if assumption \emph{(\textbf{A5})} is satisfied,
then $u$ is unique and the following estimate holds:
\begin{equation}
\alpha\int_{t_{1}}^{t_{2}}\int_{\Omega}\left\vert u^{\prime}(x,t)\right\vert
^{2}dxdt+2\sigma(u(t_{2}))-2\sigma(u(t_{1}))\leq\frac{1}{\alpha}\int_{t_{1}%
}^{t_{2}}\int_{\Omega}\left\vert f(x,t)\right\vert ^{2}dxdt \label{1.9}%
\end{equation}
for all $0\leq t_{1}\leq t_{2}\leq T$. Here $\alpha>0$ is the same as in
assumption \emph{(\textbf{A2})} and $\sigma(\cdot)$ is defined by
\emph{(\ref{1.3})}.
\end{theorem}

It is an urgent matter to make precise the comparison of our first main result
in Theorem \ref{t4.1} with the existing ones in the literature. This kind of
problem has already been considered in several work; see, e.g., \cite{Brokate,
Colli, Eleuteri, KV1994, Stefanelli, A.Visintin}, just to cite a few. Most of
these work deal with linear diffusion operators while very few treat nonlinear
cases. Although assumptions (\textbf{A1})-(\textbf{A4}) are the natural way to
generalize the linear operators (like the Laplacian) or the nonlinear ones
(like the $p$-Laplacian), to the best of our knowledge, there is no work in
the literature dealing with nonlinear PDEs with hysteresis and exhibiting such
kind of nonlinearity in the diffusion term. One of the work with assumptions
close to ours is \cite{KV1994} in which the authors considered a nonlinear
diffusion operator of the form
\[
-{\sum_{i=1}^{N}}\frac{\partial}{\partial x_{i}}a_{i}\left(  \frac{\partial
u}{\partial x_{i}}\right)  \text{ for }u\in W^{1,p}(\Omega)
\]
where the $a_{i}$'s are linear or nonlinear monotone functions defined on
$\mathbb{R}$; see \cite[p. 42]{KV1994}. We believe that one of the main
difficulties in obtaining the solutions of (\ref{eq1}) is about obtaining an
energy inequality like (\ref{eq1.7}). But this is a mere consequence of Lemma
\ref{l1.1} (see also Remark \ref{r3.1}) that stems from some properties of the
functional $J$. Hence the inequality (\ref{eq1.7}) is in order, thanks to the
monotonicity property of the hysteresis operator.

Theorem \ref{t4.1} will be proved in Section \ref{sec3}. To do this, we
proceed in several steps. First of all we have to approximate our model
problem by employing an implicit time discretization scheme of (\ref{eq14}),
which leads to the semilinear variational equation
\[
b(u_{m},\varphi)+\int_{\Omega}{b_{m}\left(  x,u_{m}(x)\right)  \varphi
(x)dx}=0\text{ for }\varphi\in W_{0}^{1,p}(\Omega),
\]
where $u_{m}=u(x,mh)$ and the functionals $b$ and $b_{m}$ are defined below in
Section \ref{sec3}. We then derive the following uniform estimates
\begin{equation}
\sum_{m=1}^{\ell}{h\left\Vert \frac{u_{m}-u_{m-1}}{h}\right\Vert
_{L^{2}(\Omega)}^{2}}+\sup_{1\leq k\leq\ell}\left\Vert \nabla u_{k}\right\Vert
_{L^{p}(\Omega)}^{p}\leq C. \label{eq03}%
\end{equation}

Next, defining the linear interpolates $u_{\ell}$ and $\tilde{u}_{\ell}$ (see 
Section \ref{sec3}), we prove that the estimate
\[
{\int_{Q}}\left\vert {{u_{\ell}^{\prime}}}\right\vert {{^{2}dx}dt}+\sup_{0\leq
t\leq T}\left(  \left\Vert \nabla\tilde{u}_{\ell}(t)\right\Vert _{L^{p}%
(\Omega)}^{p}+\left\Vert \nabla u_{\ell}(t)\right\Vert _{L^{p}(\Omega)}%
^{p}\right)  \leq C,
\]
holds uniformly in $\ell\in\mathbb{N}$, and further
\begin{equation}
\left\Vert u_{\ell}-\tilde{u}_{\ell}\right\Vert _{L^{2}(Q)}\leq\frac{C}{\ell}.
\label{eq01}%
\end{equation}
The linear interpolate $u_{\ell}$ is the approximate solution of the
discretized problem
\[
c\frac{\partial u_{\ell}}{\partial t}+\frac{\partial w_{\ell}}{\partial
t}-\Div\boldsymbol{a}(\cdot,\nabla\tilde{u}_{\ell})=\tilde{f}_{\ell}\text{ in
}W^{-1,p^{\prime}}(\Omega)\text{ a.e. in }(0,T).
\]
It is important to note that the estimate (\ref{eq01}) above allows us to
prove that the sequences $u_{\ell}$ and $\tilde{u}_{\ell}$ have the same
strong limit in $L^{2}(Q)$. This estimate replaces its counterpart in the
linear setting where the following one
\begin{equation}
\left\Vert u_{\ell}-\tilde{u}_{\ell}\right\Vert _{L^{2}(0,T;H_{0}^{1}%
(\Omega))}\leq\frac{C}{\ell} \label{eq02}%
\end{equation}
is used, enabling to conclude that the sequence $u_{\ell}-\tilde{u}_{\ell}$
strongly converges to $0$ in $L^{2}(0,T;H_{0}^{1}(\Omega))$. Estimate
(\ref{eq02}) stems from the equality
\[
\left\Vert u_{\ell}-\tilde{u}_{\ell}\right\Vert _{L^{2}(0,T;H_{0}^{1}%
(\Omega))}^{2}=\frac{T}{3\ell}\sum_{m=1}^{\ell}{\left\Vert \nabla u_{m}-\nabla
u_{m-1}\right\Vert _{L^{2}(\Omega)}^{2}}.
\]
However, in the nonlinear framework, the above equality is out of reach, and
we therefore replace it by the following one
\[
\left\Vert u_{\ell}-\tilde{u}_{\ell}\right\Vert _{L^{2}(Q)}^{2}=\frac{T}%
{3\ell}\sum_{m=1}^{\ell}{\left\Vert u_{m}-u_{m-1}\right\Vert _{L^{2}(\Omega
)}^{2},}%
\]
which, thanks to (\ref{eq03}), ensures the equality of the weak limits of both
sequences $(u_{\ell})_{\ell}$ and $(\widetilde{u}_{\ell})_{\ell}$ in
$L^{p}(0,T;W_{0}^{1,p}(\Omega))$.

In order to state the next main result, we need a further notion. We define
the $\omega$-limit set of a solution $u$ of (\ref{eq1}) by
\[
\omega(u)=\{\varphi\in W_{0}^{1,p}(\Omega):\exists t_{n}\rightarrow
+\infty\text{ such that }\lim_{n\rightarrow\infty}\left\Vert u(\cdot
,t_{n})-\varphi\right\Vert _{L^{p}(\Omega)}=0\}.
\]
It is known (see e.g. \cite[p. 1019]{Chill}) that if $u:\mathbb{R}%
_{+}\rightarrow L^{2}(\Omega)$ is a solution of (\ref{eq1}) such that the
range $\{u(\cdot,t):t\geq1\}$ is relatively compact in $L^{2}(\Omega)$, then
the $\omega$-limit set $\omega(u)$ is nonempty.

The next result is related to the existence of $\omega$-limit sets of
trajectories of (\ref{eq1}). Here, we deal with global solutions $u\in
L^{2}(\mathbb{R}_{+};L^{2}(\Omega))\cap L^{\infty}(\mathbb{R}_{+};W_{0}%
^{1,p}(\Omega))$ of (\ref{eq1}) given by Theorem \ref{t4.1} in which
assumption (\textbf{A3}) is replaced by (\textbf{A3})$_{1}$ below: 

\begin{itemize}
\item[(\textbf{A3})$_{1}$] $f\in H^{1}(\mathbb{R}_{+};L^{2}(\Omega))$ and
$u_{0}\in W_{0}^{1,p}(\Omega)$ where $\mathbb{R}_{+}=[0,\infty)$.
\end{itemize}

It is important to note that, in view of the estimate (\ref{eq1.7}) where the
constant $C$ is independent of $T$, such solutions exist by Theorem
\ref{t4.1}. At this level, we are not requiring uniqueness, but only the
existence of solutions to (\ref{eq1}).

\begin{theorem}
\label{t1.2}Assume that \emph{(\textbf{A1})}, \emph{(\textbf{A2})},
\emph{(\textbf{A3})}$_{1}$ and \emph{(\textbf{A4})} hold. Then for any $u$
given by Theorem \emph{\ref{t4.1}} and any sequence of times $(t_{n})_{n}$
such that $t_{n}\rightarrow\infty$ with $n$, there exist a subsequence of
$(t_{n})_{n}$ still denoted by $(t_{n})_{n}$ and a function $u_{\infty}\in
W_{0}^{1,p}(\Omega)$ such that
\begin{equation}
u(\cdot,t_{n})\rightarrow u_{\infty}\text{ in }L^{p}(\Omega)\text{-strong,}
\label{1.4}%
\end{equation}
where $u_{\infty}$ solves the stationary problem
\begin{equation}
-\Div\boldsymbol{a}(\cdot,\nabla u_{\infty})=g\text{ in }\Omega, \label{1.5}%
\end{equation}
the function $g$ being equal either to $0$ (if $f$ depends on the time
variable $t$) or to $f$ (if $f$ does not depend on $t$).
\end{theorem}

It is important to note that assumption (A3)$_{1}$ on $f$ entails the
continuity of $f$ with respect to $t$, so that we could define $f(\cdot,t)$
for any $t\geq0$. It is therefore made only for that purpose. It can thus be
replaced by $f\in L^{2}(\mathbb{R}_{+};L^{2}(\Omega))\cap\mathcal{C}%
(\mathbb{R}_{+};L^{2}(\Omega))$. However, our main purpose in proving the
existence of the solution of (\ref{eq1}) is to looking for the qualitative
properties of the solutions $u$ of (\ref{eq1}) under a more general assumption
on the behaviour of the source term $f$ and on the coefficient functions
$\boldsymbol{a}(\cdot,\lambda)$ with respect to both the time scale
$\tau=t/\varepsilon_{n}$ and the space scale $y=x/\varepsilon_{n}$ when the
coefficients depend on $y=x/\varepsilon_{n}$ and $\tau=t/\varepsilon_{n}$,
where $\varepsilon_{n}$ is a sequence of positive real numbers verifying
$0<\varepsilon_{n}\leq1$ with $\varepsilon_{n}\rightarrow0$ as $n\rightarrow
\infty$. This falls within the scope of homogenization theory, and depends
carefully on properties of the coefficients of the operators in (\ref{eq1}).
This is another issue which will be addressed in a very subsequent work.

\section{Some useful properties of the function $\boldsymbol{a}$$(x,\lambda)$
and a preliminary estimate\label{subsec1.2}}

We need to derive some useful properties of the function $\boldsymbol{a}$.
Since the function $J(x,\cdot)$ is convex and has a growth of order $p$ (see
in particular the right-hand side of the inequality in (\ref{0.2})) it emerges
from \cite[Proof of Theorem 2.1]{Marcellini} that
\begin{equation}
\left\vert \nabla_{\lambda}J(x,\lambda)\right\vert \leq C_{3}(1+\left\vert
\lambda\right\vert ^{p-1})\text{ for all }\lambda\in\mathbb{R}^{d}\text{, a.e.
}x\in\Omega. \label{0.3}%
\end{equation}
Indeed, since $J(x,\cdot)$ is convex and differentiable, it holds that
\begin{equation}
J(x,\lambda)-J(x,\mu)\geq\nabla_{\lambda}J(x,\mu)\cdot(\lambda-\mu)\text{ for
all }\lambda,\mu\in\mathbb{R}^{d}\text{ and a.e. }x\in\Omega. \label{0.4}%
\end{equation}
Choosing $\mu=\lambda+he_{i}$ with $h\in\mathbb{R}$ and $e_{i}$ the $i$th
vector of the canonical basis of $\mathbb{R}^{d}$, it follows that
\begin{align*}
\frac{\partial J}{\partial\lambda_{i}}(x,\lambda)  &  \leq\frac{1}%
{h}(J(x,\lambda)-J(x,\lambda+he_{i}))\text{ if }h>0\\
\frac{\partial J}{\partial\lambda_{i}}(x,\lambda)  &  \geq\frac{1}%
{h}(J(x,\lambda)-J(x,\lambda+he_{i}))\text{ if }h<0.
\end{align*}
Hence, taking $\left\vert h\right\vert =\left\vert \lambda\right\vert +1$
above and using the right-hand side of (\ref{0.2}), we are led to
\begin{align*}
\left\vert \frac{\partial J}{\partial\lambda_{i}}(x,\lambda)\right\vert  &
\leq\frac{1}{\left\vert h\right\vert }\left(  J(x,\lambda)+J(x,\lambda
+he_{i})\right) \\
&  \leq C_{2}\frac{1+(\left\vert \lambda\right\vert +1)^{p}+\left\vert
\lambda\right\vert ^{p}}{\left\vert \lambda\right\vert +1}\leq C_{3}%
(1+\left\vert \lambda\right\vert ^{p-1})
\end{align*}
since $p>1$, where $C_{3}$ depends on $C_{2}$ and $p$. We also infer from
(\ref{0.3}) that%

\[
\left\vert J(x,\lambda)-J(x,\mu)\right\vert \leq c_{2}(1+\left\vert
\lambda\right\vert ^{p-1}+\left\vert \mu\right\vert ^{p-1})\left\vert
\lambda-\mu\right\vert \text{, all }\lambda,\mu\in\mathbb{R}^{d}\text{ and
a.e. }x\in\Omega.
\]
It also follows from (\ref{0.4}) that
\[
\nabla_{\lambda}J(x,\lambda)\cdot(\lambda-t\mu)\geq J(x,\lambda)-J(x,t\mu)
\]
for all $\lambda,\mu\in\mathbb{R}^{d}$, $t\geq0$ and a.e. $x\in\Omega$.
Letting $t\rightarrow0$ and using the left-hand side of (\ref{0.2}), we get
\[
\nabla_{\lambda}J(x,\lambda)\cdot\lambda\geq c_{1}\left\vert \lambda
\right\vert ^{p}-J(x,0).
\]
Another consequence of (\ref{0.4}) is the monotonicity of $\nabla_{\lambda
}J(x,\cdot)$ expressed as follows:
\begin{equation}
(\nabla_{\lambda}J(x,\lambda)-\nabla_{\lambda}J(x,\mu))\cdot(\lambda-\mu
)\geq0\text{, all }\lambda,\mu\in\mathbb{R}^{d}\text{, a.e. }x\in\Omega.
\label{0.6}%
\end{equation}
Indeed, (\ref{0.4}) yields
\[
J(x,\lambda)-J(x,\mu)\geq\nabla_{\lambda}J(x,\mu)\cdot(\lambda-\mu)
\]
and
\[
J(x,\mu)-J(x,\lambda)\geq\nabla_{\lambda}J(x,\lambda)\cdot(\mu-\lambda).
\]
Adding these inequalities together, we obtain (\ref{0.6}).

We summarize the above properties of $\boldsymbol{a}$ here below.

\begin{itemize}
\item[(\textbf{A6})] The function $\boldsymbol{a}:(x,\lambda)\mapsto
\boldsymbol{a}(x,\lambda)$ from $\Omega\times\mathbb{R}^{d}$ to $\mathbb{R}%
^{d}$ is therefore constrained as follows:

\begin{itemize}
\item[(H)$_{1}$] For each given $\lambda\in\mathbb{R}^{d}$, the function
$x\mapsto\boldsymbol{a}(x,\lambda)$ is measurable from $\Omega$ into
$\mathbb{R}^{d}$.

\item[(H)$_{2}$] There exists a positive constant $\alpha_{3}$ such that
$\boldsymbol{a}(x,\lambda)\cdot\lambda\geq\alpha_{1}\left\vert \lambda
\right\vert ^{p}-\alpha_{3}$, where $\alpha_{3}=\left\Vert J(\cdot
,0)\right\Vert _{L^{\infty}(\Omega)}$.

\item[(H)$_{3}$] There is a constant $C_{2}>0$, such that, a.e. in $x\in
\Omega$, for $\lambda_{1},\lambda_{2}\in\mathbb{R}^{d}$,%
\[
(\boldsymbol{a}(x,\lambda_{1})-\boldsymbol{a}(x,\lambda_{2}))\cdot(\lambda
_{1}-\lambda_{2})\geq0,
\]%
\[
\left\vert \boldsymbol{a}(x,\lambda_{1})-\boldsymbol{a}(x,\lambda
_{2})\right\vert \leq C_{2}(1+\left\vert \lambda_{1}\right\vert +\left\vert
\lambda_{2}\right\vert )^{p-2}\left\vert \lambda_{1}-\lambda_{2}\right\vert ,
\]
where the dot denotes the usual Euclidean inner product in $\mathbb{R}^{d}$,
and $\left\vert \cdot\right\vert $ the associated norm.
\end{itemize}
\end{itemize}

The following result will be of interest in the sequel.

\begin{lemma}
\label{l1.1}Let $\boldsymbol{a}:\Omega\times\mathbb{R}^{d}\rightarrow
\mathbb{R}^{d}$ satisfy \emph{(\textbf{A6})} above. Assume that $u\in
L^{\infty}(0,T;W_{0}^{1,p}(\Omega))$, $u^{\prime}\in L^{2}(0,T;H_{0}%
^{1}(\Omega))$ and ${\Div}\boldsymbol{a}(\cdot,\nabla u)\in L^{p^{\prime}%
}(0,T;W^{-1,p^{\prime}}(\Omega))$. Then the function $t\mapsto\sigma
(u(t))=\int_{\Omega}J(\cdot,\nabla u(t))dx$ is absolutely continuous on
$(0,T)$ and
\begin{equation}
\frac{d}{dt}\sigma(u(t))=-\left\langle {\Div}\boldsymbol{a}(\cdot,\nabla
u(t)),u^{\prime}(t)\right\rangle \text{ for a.e. }t\in\lbrack0,T] \label{1.1}%
\end{equation}
where $u(t)=u(\cdot,t)$ and $\left\langle {\cdot},\cdot\right\rangle $ denotes
the duality pairing between $L^{p^{\prime}}(0,T;W^{-1,p^{\prime}}(\Omega))$
and $L^{p}(0,T;W_{0}^{1,p}(\Omega))$.
\end{lemma}

\begin{proof}
Let $h>0$ be arbitrarily fixed. Then using inequality (\ref{0.4}), we obtain,
for a.e. $t\in(0,T)$,
\begin{align*}
\boldsymbol{a}(\cdot,\nabla u(t))\cdot(\nabla u(t+h)-\nabla u(t))  &  \leq
J(\cdot,\nabla u(t+h))-J(\cdot,\nabla u(t))\\
&  \leq\boldsymbol{a}(\cdot,\nabla u(t+h))\cdot(\nabla u(t+h)-\nabla u(t)).
\end{align*}
Integrating the above inequalities with respect to $(x,t)$, we have, for a.e.
$0\leq t_{1}\leq t_{2}\leq T$,
\begin{equation}%
\begin{array}
[c]{l}%
-%
{\displaystyle\int_{t_{1}}^{t_{2}}}
\left\langle {\Div}\boldsymbol{a}(\cdot,\nabla u(\tau)),\frac{u(\tau
+h)-u(\tau)}{h}\right\rangle d\tau\\
=\frac{1}{h}%
{\displaystyle\int_{t_{1}}^{t_{2}}}
{\displaystyle\int_{\Omega}}
\boldsymbol{a}(\cdot,\nabla u(\tau))\cdot(\nabla u(\tau+h)-\nabla
u(\tau))dxd\tau\\
\leq\frac{1}{h}%
{\displaystyle\int_{t_{1}}^{t_{2}}}
{\displaystyle\int_{\Omega}}
J(\cdot,\nabla u(\tau+h))dxd\tau-\frac{1}{h}%
{\displaystyle\int_{t_{1}}^{t_{2}}}
{\displaystyle\int_{\Omega}}
J(\cdot,\nabla u(\tau))dxd\tau\\
\leq\frac{1}{h}%
{\displaystyle\int_{t_{1}}^{t_{2}}}
{\displaystyle\int_{\Omega}}
\boldsymbol{a}(\cdot,\nabla u(\tau+h))\cdot(\nabla u(\tau+h)-\nabla
u(\tau))dxd\tau\\
=-%
{\displaystyle\int_{t_{1}}^{t_{2}}}
\left\langle {\Div}\boldsymbol{a}(\cdot,\nabla u(\tau+h)),\frac{u(\tau
+h)-u(\tau)}{h}\right\rangle d\tau
\end{array}
\label{2*}%
\end{equation}
where $\left\langle \cdot,\cdot\right\rangle $ denotes the duality pairings
between $L^{p^{\prime}}(0,T;W^{-1,p^{\prime}}(\Omega))$ and $L^{p}%
(0,T;W_{0}^{1,p}(\Omega))$. But
\[%
{\displaystyle\int_{t_{1}}^{t_{2}}}
{\displaystyle\int_{\Omega}}
J(\cdot,\nabla u(\tau+h))dxd\tau=%
{\displaystyle\int_{t_{1}+h}^{t_{2}+h}}
{\displaystyle\int_{\Omega}}
J(\cdot,\nabla u(\tau))dxd\tau
\]
and $%
{\displaystyle\int_{t_{1}+h}^{t_{2}+h}}
-%
{\displaystyle\int_{t_{1}}^{t_{2}}}
=%
{\displaystyle\int_{t_{2}}^{t_{2}+h}}
-%
{\displaystyle\int_{t_{1}}^{t_{1}+h}}
$, so that (\ref{2*}) becomes
\begin{equation}%
\begin{array}
[c]{l}%
-%
{\displaystyle\int_{t_{1}}^{t_{2}}}
\left\langle {\Div}\boldsymbol{a}(\cdot,\nabla u(\tau)),\frac{u(\tau
+h)-u(\tau)}{h}\right\rangle d\tau\\
\leq\frac{1}{h}%
{\displaystyle\int_{t_{2}}^{t_{2}+h}}
{\displaystyle\int_{\Omega}}
J(\cdot,\nabla u(\tau))dxd\tau-\frac{1}{h}%
{\displaystyle\int_{t_{1}}^{t_{1}+h}}
{\displaystyle\int_{\Omega}}
J(\cdot,\nabla u(\tau))dxd\tau\\
\leq-%
{\displaystyle\int_{t_{1}}^{t_{2}}}
\left\langle {\Div}\boldsymbol{a}(\cdot,\nabla u(\tau+h)),\frac{u(\tau
+h)-u(\tau)}{h}\right\rangle d\tau.
\end{array}
\label{3*}%
\end{equation}
We recall that all the integrals involved in (\ref{3*}) are well defined
according to the assumption (\textbf{A6}) and the fact that $u\in L^{\infty
}(0,T;W_{0}^{1,p}(\Omega))$. Letting $h\rightarrow0$ in (\ref{3*}) (where we
use the other assumptions in Lemma \ref{l1.1}) yields, for a.e. $t_{1},t_{2}%
$,
\begin{align*}
-%
{\displaystyle\int_{t_{1}}^{t_{2}}}
\left\langle {\Div}\boldsymbol{a}(\cdot,\nabla u(\tau)),u^{\prime}%
(\tau)\right\rangle d\tau &  =%
{\displaystyle\int_{\Omega}}
J(\cdot,\nabla u(t_{2}))dx-%
{\displaystyle\int_{\Omega}}
J(\cdot,\nabla u(t_{1}))dx\\
&  =\sigma(u(t_{2}))-\sigma(u(t_{1})),
\end{align*}
thereby showing that the mapping $t\mapsto\sigma(u(t))$ is absolutely
continuous on $(0,T)$ and that (\ref{1.1}) is satisfied.
\end{proof}

\begin{remark}
\label{r3.1}\emph{Note that Lemma \ref{l1.1} remains true if the assumptions }%
\[
u^{\prime}\in L^{2}(0,T;H_{0}^{1}(\Omega))\text{\emph{ and }}{\Div}%
\boldsymbol{a}(\cdot,\nabla u)\in L^{p^{\prime}}(0,T;W^{-1,p^{\prime}}%
(\Omega))
\]
\emph{are replaced by the following ones therein: }%
\begin{equation}
u^{\prime}\in L^{2}(0,T;L^{2}(\Omega))\text{\emph{ and }}{\Div}\boldsymbol{a}%
(\cdot,\nabla u)\in L^{2}(0,T;L^{2}(\Omega)),\label{1.8}%
\end{equation}
\emph{the other ones remaining unchanged. In that case, the duality pairings
}$\left\langle \cdot,\cdot\right\rangle $\emph{ will be replaced by the inner
product in }$L^{2}(0,T;L^{2}(\Omega))$\emph{ and we will proceed by
approximation like in \cite[Proposition 2.11]{Brezis2} to obtain (\ref{1.1})
for the approximating sequence and conclude like in \cite[Lemma 3.3]{Brezis2}
after a limit passage.}
\end{remark}

This being so, to see what regularity can be expected for a solution to
(\ref{eq1}), we first present an informal argument. We test (\ref{eq1}) by
$u^{\prime}=\frac{\partial u}{\partial t}$ and integrate over $\Omega$ to
obtain, for $t>0$,%

\begin{equation}
\int_{\Omega}{c}\left\vert {u}^{\prime}{(t)}\right\vert ^{2}{dx}+\int_{\Omega
}{u}^{\prime}(t){w}^{\prime}{(t)dx}-{\left\langle {\Div}\boldsymbol{a}%
(\cdot,\nabla u(t)),u^{\prime}(t)\right\rangle }=\int_{\Omega}{u}^{\prime
}{(t)f(t)dx} \label{eq1.6}%
\end{equation}
where we have used the abbreviation $u(t)=u(\cdot,t)$. Assuming $\mathcal{W}%
[\cdot;x]$ is piecewise monotone for every $x\in\Omega$, then we have
$u^{\prime}w^{\prime}\geq0$, so that the second term of the left-hand side of
(\ref{eq1.6}) becomes non-negative, \textit{i.e.} $\int_{\Omega}{u}^{\prime
}(t){w}^{\prime}{(t)dx}\geq0$. Since (see (\ref{1.1}))
\[
-\left\langle {\Div}\boldsymbol{a}(\cdot,\nabla u(t)),u^{\prime}%
(t)\right\rangle =\frac{d}{dt}\sigma(u(t))
\]
where
\begin{equation}
\sigma(u(t))=\int_{\Omega}J(x,\nabla{u(x,t)})dx, \label{1.3}%
\end{equation}
we integrate (\ref{eq1.6}) with respect to $t$ and apply appropriate Young's
inequality to its right-hand side to get
\[
\alpha\int_{0}^{t}\int_{\Omega}\left\vert u^{\prime}(\tau)\right\vert
^{2}dxd\tau+\sigma(u(t))-\sigma(u_{0})\leq\frac{1}{2\alpha}\int_{0}^{t}%
\int_{\Omega}\left\vert f\right\vert ^{2}dxd\tau+\frac{\alpha}{2}\int_{0}%
^{t}\int_{\Omega}\left\vert u^{\prime}(\tau)\right\vert ^{2}dxd\tau,
\]
where we have also used the inequality $c\geq\alpha$ in assumption
(\textbf{A2}). Using the left-hand side of inequality (\ref{0.2}), we infer
\[
\int_{0}^{t}\int_{\Omega}\left\vert u^{\prime}(t)\right\vert ^{2}%
dxd\tau+\left\Vert \nabla u(t)\right\Vert _{L^{p}(\Omega)}^{p}\leq C\int
_{0}^{T}\int_{\Omega}\left\vert f\right\vert ^{2}dxdt+C\left\Vert \nabla
u_{0}\right\Vert _{L^{p}(\Omega)}^{p}%
\]
where $C$ depends only on $\alpha$, $\alpha_{1}$ and $p$. Hence we find the a
priori estimate%

\begin{equation}
\int_{0}^{t}\int_{\Omega}\left\vert u^{\prime}(t)\right\vert ^{2}%
dxdt+\sup_{0\leq t\leq T}\left\Vert \nabla u(t)\right\Vert _{L^{p}(\Omega
)}^{p}\leq C\int_{0}^{T}\int_{\Omega}\left\vert f\right\vert ^{2}%
dxdt+C\left\Vert \nabla u_{0}\right\Vert _{L^{p}(\Omega)}^{p}. \label{eq1.7}%
\end{equation}
Since $f\in L^{2}(0,T;L^{2}(\Omega))\equiv L^{2}(Q)$ and $u_{0}\in W_{0}%
^{1,p}(\Omega)$, we are thus led to look for weak solutions $u$ in the space
\[
Y=L^{\infty}(0,T;W_{0}^{1,p}(\Omega))\cap H^{1}\left(  0,T;L^{2}%
(\Omega)\right)  ,
\]
which implies $u\in\mathcal{C}([0,T],L^{2}(\Omega))$ thanks to the continuous
imbedding
\[
Y\hookrightarrow\mathcal{C}([0,T],L^{2}(\Omega)).
\]
Since $u(x,\cdot)$ is the input of the hysteresis operator $\mathcal{W}%
[\cdot;x]$ at each space point $x\in\Omega$, the compactness of the imbedding%

\begin{equation}
Y\hookrightarrow L^{2}(\Omega;\mathcal{C}\left(  [0,T]\right)  ) \label{eq8'}%
\end{equation}
will play a crucial role in the existence proof if the hysteresis operators
$\mathcal{W}[\cdot;x]$ are continuous on $\mathcal{C}\left(  [0,T]\right)  $.
The proof of the compactness of the embedding (\ref{eq8'}) is given in
\cite[Corollary 3.2.3]{Brokate} for $d=1$; for $d>1$, we follow the argument
of Visintin based on interpolation theory, and obtain the result from the
chain of continuous imbeddings%

\begin{equation}
Y\hookrightarrow H^{1}(Q)\hookrightarrow H^{\sigma}\left(  \Omega;H^{1-\sigma
}(0,T)\right)  \hookrightarrow L^{2}\left(  \Omega;\mathcal{C}\left(
[0,T]\right)  \right)  \label{eq9'}%
\end{equation}
for $0<\sigma<\frac{1}{2}$, where the last imbedding is compact.

\section{Proof of Theorem \ref{t4.1}: Existence result\label{sec3}}

For the proof of the existence result, we proceed in three steps listed in the
following subsections.

\subsection{\textbf{Approximation and existence of approximate solutions}%
\label{subsec3.1.1}}

We have to approximate our model problem by employing an implicit time
discretization scheme of (\ref{eq14}). To this end, let $\ell\in
\mathbb{N}\backslash\{0\}$ ($\mathbb{N}$ the set of nonnegative integers) be
given, and set $h=T/\ell$ where $T$ is the final time. In the sequel, we will
denote by $C$, some positive constants that may depend on $\Omega$, $T$, $f$,
and the initial data, but neither on $\ell$ and nor on $m\in\left\{
1,...,\ell\right\}  $.

For $1\leq m\leq\ell$, we consider the semidiscrete problem on the time level
$t=mh$ for the unknown functions $u_{m},w_{m}:\Omega\rightarrow\mathbb{R}$
given by
\begin{equation}%
\begin{array}
[c]{l}%
\frac{1}{h}%
{\displaystyle\int_{\Omega}}
{c(u_{m}-u_{m-1})\varphi dx}+\frac{1}{h}%
{\displaystyle\int_{\Omega}}
{(w_{m}-w_{m-1})\varphi dx}+%
{\displaystyle\int_{\Omega}}
{\boldsymbol{a}(x,\nabla u_{m})\cdot\nabla\varphi dx}\\
\ \ \ \ \ \ \ \ \ =%
{\displaystyle\int_{\Omega}}
{f_{m}\varphi dx}\text{ for all }\varphi\in W_{0}^{1,p}(\Omega)
\end{array}
\label{eq15}%
\end{equation}%
\begin{equation}
w_{m}(x)=\mathcal{W}_{f}((u_{0}(x),\ldots,u_{m}(x));x)\ \text{for a.e }%
x\in\Omega\label{eq16}%
\end{equation}
where $u_{0}(x)=u(x,0)$ is given by Assumption (\textbf{A3}) and
\begin{equation}
u_{m}(x)=u(x,mh)\text{ and }f_{m}(x)=\frac{1}{h}\int_{(m-1)h}^{mh}%
{f(x,t)dt},\ \ w_{0}(x)=\mathcal{W}_{f}\left(  u_{0}(x);x\right)  .
\label{eq17}%
\end{equation}

We can rewrite (\ref{eq15}) in the following form:
\begin{equation}
c\frac{u_{m}-u_{m-1}}{h}+\frac{w_{m}-w_{m-1}}{h}-{\Div}\boldsymbol{a}(x,\nabla
u_{m})=f_{m}\text{ in }W^{-1,p^{\prime}}(\Omega). \label{eq1916}%
\end{equation}

We rewrite (\ref{eq15}) as a semilinear variational equation,
\begin{equation}
b(u_{m},\varphi)+\int_{\Omega}{b_{m}\left(  x,u_{m}(x)\right)  \varphi
(x)dx}=0\text{ for all }\varphi\in W_{0}^{1,p}(\Omega) \label{eq1900}%
\end{equation}
where
\[
b(u,\varphi)=\int_{\Omega}{\boldsymbol{a}(x,\nabla u)\cdot\nabla\varphi
dx}\text{ for }u,\varphi\in W_{0}^{1,p}(\Omega)
\]
and the function $b_{m}:\Omega\times\mathbb{R}\rightarrow\mathbb{R}$, which is
defined by
\[%
\begin{array}
[c]{l}%
b_{m}(x,u)=\frac{1}{h}\left(  cu+\mathcal{W}_{f}\left(  (u_{0}(x),...,u_{m-1}%
(x)),u\right)  ;x\right) \\
\ \ \ \ \ \ \ \ \ \ \ \ \ \ \ \ \ \ \ \ -\frac{1}{h}\left(  cu_{m-1}%
(x)+w_{m-1}(x)\right)  -f_{m}(x),
\end{array}
\]
is measurable in $x$ and continuous in $u$. Moreover $b_{m}(x,\cdot)$ is
strictly increasing; indeed $\mathcal{W}[\cdot;x]$ is piecewise increasing for
each fixed $x\in\Omega$. Finally, we infer from (\ref{eq1.11}) that, for all
$(x,u)\in\Omega\times\mathbb{R}$,%
\[
\left\vert b_{m}(x,u)\right\vert \leq\kappa_{2}(x)+C\left(  \left\vert
w_{m-1}(x)\right\vert +\sum_{k=0}^{m-1}{\left\vert u_{k}(x)\right\vert
}\right)  +C\left\vert u\right\vert ,
\]
with a suitable positive constants $C$ and some function $\kappa_{2}\in
L^{2}(\Omega)$. Considering the induction over $m$ (where we have used
\cite[Theorem 1.3.2]{Brokate} for the induction step) associated to Theorem
\ref{t2.4}, we derive, for each $m\in\left\{  1,...,\ell\right\}  $, the
existence of a unique $u_{m}\in W_{0}^{1,p}(\Omega)$ solution to
(\ref{eq1900}). Moreover, the function $w_{m}$, as defined by (\ref{eq16}),
belongs to $L^{2}(\Omega)$.

With the sequence $(u_{m},w_{m})_{m}$ in hands, the next step is to find
appropriate uniform estimates which will be used in order to pass to the limit.

\subsection{\textbf{Uniform estimates}\label{subsec3.1.2}}

The goal here is first to derive the discrete version of (\ref{eq1.7}), and
next to apply it to obtain a continuous version similar to (\ref{eq1.7}), but
for a sequence of linear interpolates of $u_{m}$. The first result reads as follows.

\begin{lemma}
\label{l3.1}Let $u_{m}$ be defined by \emph{(\ref{eq1900})}. Then one has
\begin{equation}
\sum_{m=1}^{\ell}{h\left\Vert \frac{u_{m}-u_{m-1}}{h}\right\Vert
_{L^{2}(\Omega)}^{2}}+\sup_{1\leq k\leq\ell}\left\Vert \nabla u_{k}\right\Vert
_{L^{p}(\Omega)}^{p}\leq C \label{3.3}%
\end{equation}
for all positive integer $\ell$, where $C=C(p,u_{0},\Omega,\alpha,\alpha
_{1},\alpha_{2})$.
\end{lemma}

\begin{proof}
Let the integer $m\geq1$ be fixed. We insert $\varphi=u_{m}-u_{m-1}$ in
(\ref{eq15}) and we have
\begin{equation}%
\begin{array}
[c]{l}%
\frac{1}{h}%
{\displaystyle\int_{\Omega}}
{c(u_{m}-u_{m-1})^{2}dx}+\frac{1}{h}%
{\displaystyle\int_{\Omega}}
{(w_{m}-w_{m-1})(u_{m}-u_{m-1})dx}\\
\ \ \ \ \ \ +%
{\displaystyle\int_{\Omega}}
{\boldsymbol{a}(x,\nabla u_{m})\cdot(\nabla u_{m}-\nabla u_{m-1})dx}=%
{\displaystyle\int_{\Omega}}
{f_{m}(u_{m}-u_{m-1})dx.}%
\end{array}
\label{eq1901}%
\end{equation}
Since $\mathcal{W}[\cdot;x]$ is piecewise increasing for every $x\in\Omega$,
it holds $\left(  u_{m}-u_{m-1}\right)  \left(  w_{m}-w_{m-1}\right)  \geq0$
and so, using the fact that the second term of the left-hand side of
(\ref{eq1901}) is non-negative and using Assumption (\textbf{A2}) we obtain
\begin{equation}
\frac{1}{h}\alpha\left\Vert u_{m}-u_{m-1}\right\Vert _{L^{2}(\Omega)}^{2}%
+\int_{\Omega}{\boldsymbol{a}(x,\nabla u_{m})\cdot\left(  \nabla u_{m}-\nabla
u_{m-1}\right)  dx}\leq\int_{\Omega}{f_{m}\left(  u_{m}-u_{m-1}\right)  dx}.
\label{eq21}%
\end{equation}
We sum both sides of (\ref{eq21}) from $m=1$ to $m=k$ (where $1\leq k\leq\ell
$) to get
\begin{equation}
\sum_{m=1}^{k}\frac{1}{h}\alpha\left\Vert u_{m}-u_{m-1}\right\Vert
_{L^{2}(\Omega)}^{2}+\sum_{m=1}^{k}\int_{\Omega}{\boldsymbol{a}(x,\nabla
u_{m})\cdot\left(  \nabla u_{m}-\nabla u_{m-1}\right)  dx}\leq\sum_{m=1}%
^{k}\int_{\Omega}{f_{m}\left(  u_{m}-u_{m-1}\right)  dx}. \label{eq1902}%
\end{equation}
Using Cauchy-Schwarz's inequality on the right-hand side of (\ref{eq1902}), we
obtain
\begin{equation}
\sum_{m=1}^{k}{\int_{\Omega}{f_{m}\left(  u_{m}-u_{m-1}\right)  dx}}%
\leq\left(  \sum_{m=1}^{k}{h\left\Vert f_{m}\right\Vert _{L^{2}(\Omega)}^{2}%
}\right)  ^{\frac{1}{2}}\left(  \sum_{m=1}^{k}{h\left\Vert \frac{u_{m}%
-u_{m-1}}{h}\right\Vert _{L^{2}(\Omega)}^{2}}\right)  ^{\frac{1}{2}}.
\label{eq22}%
\end{equation}
Note that
\begin{equation}
\sum_{m=1}^{k}{h\left\Vert f_{m}\right\Vert _{L^{2}(\Omega)}^{2}}\leq\int
_{0}^{T}{\int_{\Omega}{\left\vert f(x,t)\right\vert ^{2}dx}dt}\leq C.
\label{eq24}%
\end{equation}
Hence we have the inequality
\begin{equation}
\sum_{m=1}^{k}{\int_{\Omega}{f_{m}\left(  u_{m}-u_{m-1}\right)  dx}}\leq
C\left(  \sum_{m=1}^{k}{h\left\Vert \frac{u_{m}-u_{m-1}}{h}\right\Vert
_{L^{2}(\Omega)}^{2}}\right)  ^{\frac{1}{2}}. \label{eq1903}%
\end{equation}
Now we have to deal with the second term of the left-hand side of
(\ref{eq1902}). To this end, we first recall that, according to the definition
of the function $\boldsymbol{a}$ given in Assumption (\textbf{A3}), we have
\begin{equation}
\boldsymbol{a}(x,\nabla u_{m})\cdot(\nabla u_{m}-\nabla u_{m-1})=\nabla
_{\lambda}J(x,\nabla u_{m})\cdot(\nabla u_{m}-\nabla u_{m-1}). \label{eq1904}%
\end{equation}
It follows from (\ref{0.4}) that
\begin{equation}
J(x,\nabla u_{m})-J(x,\nabla u_{m-1})\leq\boldsymbol{a}(x,\nabla u_{m}%
)\cdot(\nabla u_{m}-\nabla u_{m-1}). \label{eq1906}%
\end{equation}
Hence
\begin{equation}
\sum_{m=1}^{k}\left[  J(x,\nabla u_{m})-J(x,\nabla u_{m-1})\right]  \leq
\sum_{m=1}^{k}\boldsymbol{a}(x,\nabla u_{m})\cdot(\nabla u_{m}-\nabla
u_{m-1}), \label{eq1907}%
\end{equation}
that is,
\begin{equation}
J(x,\nabla u_{k})-J(x,\nabla u_{0})\leq\sum_{m=1}^{k}\boldsymbol{a}(x,\nabla
u_{m})\cdot\left(  \nabla u_{m}-\nabla u_{m-1}\right)  . \label{eq23}%
\end{equation}
Integrating (\ref{eq23}) over $\Omega$ gives
\begin{equation}
\int_{\Omega}{\left(  J(x,\nabla u_{k})-J(x,\nabla u_{0})\right)  dx}\leq
{\sum_{m=1}^{k}\int_{\Omega}\boldsymbol{a}(x,\nabla u_{m})\cdot\left(  \nabla
u_{m}-\nabla u_{m-1}\right)  dx.} \label{eq1908}%
\end{equation}
In view of (\ref{0.2}) it holds that
\[
J(x,\nabla u_{0})\leq\alpha_{2}(1+\left\vert \nabla u_{0}\right\vert
^{p})\text{ and }\alpha_{1}\left\vert \nabla u_{k}\right\vert ^{p}\leq
J(x,\nabla u_{k})\text{,}%
\]
so that
\begin{equation}
\alpha_{1}\left\Vert \nabla u_{k}\right\Vert _{L^{p}(\Omega)}^{p}\leq
\int_{\Omega}\left(  {\sum_{m=1}^{k}\boldsymbol{a}(x,\nabla u_{m})\cdot\left(
\nabla u_{m}-\nabla u_{m-1}\right)  }\right)  {dx+\alpha_{2}}\left(
\left\vert \Omega\right\vert {+\left\Vert \nabla u_{0}\right\Vert
_{L^{p}(\Omega)}^{p}}\right)  \label{eq1910}%
\end{equation}
where $\left\vert \Omega\right\vert ={\mathrm{meas}(\Omega)}$ is the Lebesgue
measure of $\Omega$. Summarizing (\ref{eq1902}) to (\ref{eq1910}), we obtain
\begin{equation}%
\begin{array}
[c]{l}%
\sum_{m=1}^{k}\frac{1}{h}\alpha\left\Vert u_{m}-u_{m-1}\right\Vert
_{L^{2}(\Omega)}^{2}+\alpha_{1}\left\Vert \nabla u_{k}\right\Vert
_{L^{p}(\Omega)}^{p}\\
\ \ \ \leq C\left(  \sum_{m=1}^{k}{h\left\Vert \frac{u_{m}-u_{m-1}}%
{h}\right\Vert _{L^{2}(\Omega)}^{2}}\right)  ^{\frac{1}{2}}+\alpha
_{2}(\left\vert \Omega\right\vert +\left\Vert \nabla u_{0}\right\Vert
_{L^{p}(\Omega)}^{p}).
\end{array}
\label{eq1911}%
\end{equation}
Applying Young's inequality to the first term of the right-hand side of
(\ref{eq1911}), we get
\[
C\left(  \sum_{m=1}^{k}{h\left\Vert \frac{u_{m}-u_{m-1}}{h}\right\Vert
_{L^{2}(\Omega)}^{2}}\right)  ^{\frac{1}{2}}\leq\frac{\alpha}{2}\sum
_{m=1}^{\ell}{h\left\Vert \frac{u_{m}-u_{m-1}}{h}\right\Vert _{L^{2}(\Omega
)}^{2}+C.}%
\]
Coming back to (\ref{eq1911}), we end up with the following inequality
\begin{equation}
\sum_{m=1}^{\ell}{h\left\Vert \frac{u_{m}-u_{m-1}}{h}\right\Vert
_{L^{2}(\Omega)}^{2}}+\sup_{1\leq k\leq\ell}\left\Vert \nabla u_{k}\right\Vert
_{L^{p}(\Omega)}^{p}\leq C \label{eq25}%
\end{equation}
where $C=C(p,u_{0},\Omega,\alpha,\alpha_{1},\alpha_{2})>0$.
\end{proof}

Now, we have to define the linear interpolates. In order to emphasize the
dependence on $\ell$ of the sequence $(u_{m})_{m}$, we denote by $u_{{m}%
,{\ell}}$ and $w_{{m},{\ell}}$ respectively the solutions of (\ref{eq15}) and
(\ref{eq16}) for any $\ell\in\mathbb{N}$, and by $f_{{m},{\ell}}$ the averages
defined in (\ref{eq17}) for $0\leq m\leq\ell$. We define the piecewise linear
interpolates as follows:
\[
u_{\ell}\left(  x,(m+\tau)h\right)  =\tau u_{{m+1,\ell}}(x)+(1-\tau
)u_{{m},{\ell}}(x),\quad\tau\in\left[  0,1\right]
\]%
\begin{equation}
w_{\ell}\left(  x,(m+\tau)h\right)  =\tau w_{{m+1,\ell}}(x)+(1-\tau
)w_{{m},{\ell}}(x),\quad\tau\in\left[  0,1\right]  , \label{eq27}%
\end{equation}
as well as the constant interpolates
\[
\tilde{u}_{\ell}\left(  x,(m+\tau)h\right)  =u_{{m+1,\ell}}(x);\ \tilde
{f}_{\ell}\left(  x,(m+\tau)h\right)  =f_{{m+1,\ell}}(x),\ \tau\in\left[
0,1\right]
\]
for $0\leq m\leq\ell-1$. With the above notation, (\ref{eq15}) reads as
\begin{equation}%
\begin{array}
[c]{l}%
{\displaystyle\int_{\Omega}}
{cu_{\ell}^{\prime}(x,t)\varphi(x)dx}+%
{\displaystyle\int_{\Omega}}
{w_{\ell}^{\prime}(x,t)\varphi(x)dx}+%
{\displaystyle\int_{\Omega}}
{\boldsymbol{a}(x,\nabla\tilde{u}_{\ell})\cdot\nabla\varphi(x)dx}\\
\ \ \ \ \ =%
{\displaystyle\int_{\Omega}}
{\tilde{f}_{\ell}(x,t)\varphi(x)dx}\ \ \text{for all }\varphi\in W_{0}%
^{1,p}(\Omega)\text{ and a.e. }t\in(0,T).
\end{array}
\label{eq29}%
\end{equation}
Thus (\ref{eq1916}) becomes
\begin{equation}
c\frac{\partial u_{\ell}}{\partial t}+\frac{\partial w_{\ell}}{\partial
t}-{\Div}\boldsymbol{a}(x,\nabla\tilde{u}_{\ell})=\tilde{f}_{\ell}\text{ in
}W^{-1,p^{\prime}}(\Omega)\text{ a.e. in }(0,T). \label{eq1917}%
\end{equation}

\begin{lemma}
\label{l3.2}Let $u_{\ell}$, $w_{\ell}$ and $\tilde{u}_{\ell}$ satisfying
\emph{(\ref{eq29})}. Then there exists a positive constant $C$ independent of
$\ell$ such that
\begin{equation}
\left\Vert {{u_{\ell}^{\prime}}}\right\Vert _{L^{2}(Q)}^{2}+\sup_{0\leq t\leq
T}\left(  \left\Vert \nabla\tilde{u}_{\ell}(t)\right\Vert _{L^{p}(\Omega)}%
^{p}+\left\Vert \nabla u_{\ell}(t)\right\Vert _{L^{p}(\Omega)}^{p}\right)
\leq C,\label{eq2000}%
\end{equation}%
\[
\left\Vert w_{\ell}\right\Vert _{L^{2}(Q)}\leq C,\ \left\Vert \boldsymbol{a}%
(\cdot,\nabla\tilde{u}_{\ell})\right\Vert _{L^{p^{\prime}}(Q)}\leq C,
\]
and
\[
\left\Vert \frac{\partial}{\partial t}\left(  cu_{\ell}+w_{\ell}\right)
\right\Vert _{L^{p^{\prime}}(0,T;W^{-1,p^{\prime}}(\Omega))}\leq C
\]
for all $\ell\in\mathbb{N}$.
\end{lemma}

\begin{proof}
First of all, it follows from (\ref{eq29}) that
\begin{equation}
\left\{
\begin{array}
[c]{l}%
{\displaystyle\int_{Q}}
({{cu_{\ell}^{\prime}+w_{\ell}^{\prime})\varphi dx}dt}+{%
{\displaystyle\int_{Q}}
{\boldsymbol{a}(\cdot,\nabla\tilde{u}_{\ell})\cdot\nabla\varphi dx}dt}={%
{\displaystyle\int_{Q}}
{\tilde{f}_{\ell}\varphi dx}dt}\\
\ \text{for all }\varphi\in L^{p}(0,T;W_{0}^{1,p}(\Omega)).
\end{array}
\right.  \label{eq30}%
\end{equation}
Proceeding exactly (multiply (\ref{eq1917}) by ${{u_{\ell}^{\prime}}}$ and
integrate over $\Omega\times(0,T)$) as we did in Section \ref{subsec1.2} to
obtain the estimate (\ref{eq1.7}), we get mutatis mutandis:
\[%
{\displaystyle\int_{Q}}
\left\vert u_{\ell}^{\prime}\right\vert ^{2}dxdt+\sup_{0\leq t\leq
T}\left\Vert \nabla\tilde{u}_{\ell}(t)\right\Vert _{L^{p}(\Omega)}^{p}\leq C%
{\displaystyle\int_{Q}}
{\left\vert \tilde{f}_{\ell}\right\vert ^{2}dxdt}+C\left\Vert \nabla
u_{0}\right\Vert _{L^{p}(\Omega)}^{p}.
\]

By virtue of the estimate (\ref{eq25}) we obtain:
\begin{equation}%
{\displaystyle\int_{Q}}
\left\vert {{u_{\ell}^{\prime}}}\right\vert {{^{2}dx}dt}+\sup_{0\leq t\leq
T}\left(  \left\Vert \nabla\tilde{u}_{\ell}(t)\right\Vert _{L^{p}(\Omega)}%
^{p}+\left\Vert \nabla u_{\ell}(t)\right\Vert _{L^{p}(\Omega)}^{p}\right)
\leq C.
\end{equation}

Using the inequality (\ref{0.3}), we get
\[
\left\Vert \boldsymbol{a}(\cdot,\nabla u_{\ell})\right\Vert _{L^{p^{\prime}%
}(Q)}^{p^{\prime}}\leq C(1+\left\Vert \nabla u_{\ell}\right\Vert _{L^{p}%
(Q)}^{p})\leq C
\]
where $C$ is a positive constant depending on $\mathrm{meas}(\Omega)$ and $T$,
but not on $\ell$.

Also
\begin{equation}
\left\Vert u_{\ell}\right\Vert _{L^{p}(0,T;W_{0}^{1,p}(\Omega))\cap
H^{1}(0,T;L^{2}(\Omega))}\leq C\label{eq1912}%
\end{equation}
and
\begin{equation}
\left\Vert \tilde{u}_{\ell}\right\Vert _{L^{\infty}(0,T;W_{0}^{1,p}%
(\Omega))\cap H^{1}(0,T;L^{2}(\Omega))}\leq C\label{eq1913}%
\end{equation}
The inequality (\ref{0.3}) associated to (\ref{eq1913}) yield
\[
\left\Vert \boldsymbol{a}(\cdot,\nabla\tilde{u}_{\ell})\right\Vert
_{L^{p^{\prime}}(Q)}\leq C.
\]
Finally we find from (\ref{eq30}) and (\ref{eq2000}) that for any $\varphi\in
L^{p}(0,T;W_{0}^{1,p}(\Omega))$,
\begin{equation}
\left\vert {\int_{Q}{w_{\ell}^{\prime}\varphi dx}dt}\right\vert \leq
C\left\Vert \varphi\right\Vert _{L^{p}(0,T;W_{0}^{1,p}(\Omega))},\label{eq33}%
\end{equation}
so that
\begin{equation}
\left\Vert w_{\ell}^{\prime}\right\Vert _{L^{p^{\prime}}(0,T;W^{-1,p^{\prime}%
}(\Omega))}\leq C.\label{eq34}%
\end{equation}
According to assumption (\textbf{A4}), $\mathcal{W}$ is affine bounded, i.e.
there exist $L>0$ and $\upsilon\in L^{2}(\Omega)$ such that for any measurable
function $u:\Omega\rightarrow\mathcal{C}([0,T])$ we have
\begin{equation}
\left\Vert \mathcal{W}(u)(x,\cdot)\right\Vert _{\mathcal{C}([0,T])}\leq
L\left\Vert u(x,\cdot)\right\Vert _{\mathcal{C}([0,T])}+\upsilon(x)\text{ a.e.
in }\Omega,\label{eq1918}%
\end{equation}
and using (\ref{eq1912}) and (\ref{eq1918}), we get
\[
\left\Vert w_{\ell}\right\Vert _{L^{2}(Q)}\leq\sqrt{T}\left\Vert w_{\ell
}\right\Vert _{L^{2}(\Omega;\mathcal{C}([0,T]))}\leq\sqrt{T}L\left\Vert
u_{\ell}\right\Vert _{L^{2}(\Omega;\mathcal{C}([0,T]))}+\sqrt{T}\left\Vert
\upsilon\right\Vert _{L^{2}(\Omega)}\leq C.
\]
where $C>0$ is independent of $\ell$. So we obtain
\[
\left\Vert w_{\ell}\right\Vert _{L^{2}(Q)}\leq C.
\]
The same reasoning as in (\ref{eq33}) yields
\[
\left\Vert \frac{\partial}{\partial t}\left(  cu_{\ell}+w_{\ell}\right)
\right\Vert _{L^{p^{\prime}}(0,T;W^{-1,p^{\prime}}(\Omega))}\leq C.
\]

\end{proof}

\subsection{\textbf{Passage to the limit}\label{subsec3.1.3}}

Our goal here is to pass to the limit in each term of the variational
formulation (\ref{eq30}).

The a priori estimates we found in Lemma \ref{l3.2} allow us to conclude that,
by a standard compactness result which can be found in \cite{Lions}, e.g., the
sequence $(u_{\ell})_{\ell}$ stays in a compact subset of $L^{2}(Q)$. Invoking
some well-known results, we derive the existence of $u,\widetilde{u}\in
L^{\infty}(0,T;W_{0}^{1,p}(\Omega)),\mathbf{v}\in L^{p^{\prime}}(Q)^{N}$,
$V\in L^{p^{\prime}}(0,T;W^{-1,p^{\prime}}(\Omega))$ and $w\in L^{2}(Q)$ such
that, up to a subsequence not relabeled, we have
\[
u_{\ell}\rightarrow u\text{ in }L^{p}(0,T;W_{0}^{1,p}(\Omega))\text{-weak and
in }L^{\infty}(0,T;W_{0}^{1,p}(\Omega))\text{-weak}\ast
\]%
\begin{equation}
u_{\ell}\rightarrow u\text{ in }{L^{2}(Q)}\text{-strong}\label{eq1920}%
\end{equation}%
\[
\tilde{u}_{\ell}\rightarrow\widetilde{u}\text{ in }L^{p}(0,T;W_{0}%
^{1,p}(\Omega){)}\text{-weak and in }L^{\infty}(0,T;W_{0}^{1,p}{(\Omega
))}\text{-weak}\ast
\]%
\begin{equation}
w_{\ell}\rightarrow w\text{ in }L^{2}(Q)\text{-weak}\label{eq37}%
\end{equation}%
\[
\frac{\partial u_{\ell}}{\partial t}\rightarrow\frac{\partial u}{\partial
t}\text{ in }L^{2}(Q)\text{-weak}%
\]%
\[
\boldsymbol{a}(\cdot,\nabla\widetilde{u}_{\ell})\rightarrow\mathbf{v}\text{ in
}L^{p^{\prime}}(Q)^{N}\text{-weak}%
\]%
\[
\frac{\partial}{\partial t}(cu_{\ell}+w_{\ell})\rightarrow V\text{ in
}L^{p^{\prime}}(0,T;W^{-1,p^{\prime}}(\Omega))\text{-weak.}%
\]
It follows readily from (\ref{eq1920}) that $cu_{\ell}\rightarrow cu$ in
$L^{2}(Q)$-strong, so that, appealing to (\ref{eq37}), we get at once
\[
V=\frac{\partial}{\partial t}(cu+w).
\]
We deduce from (\ref{eq34}) that $w^{\prime}\in L^{p^{\prime}}%
(0,T;W^{-1,p^{\prime}}(\Omega))$ with $w_{\ell}^{\prime}\rightarrow w^{\prime
}$ in $L^{p^{\prime}}(0,T;W^{-1,p^{\prime}}(\Omega))$-weak. Hence
\[
w\in L^{2}(Q)\text{ with }w^{\prime}\in L^{p^{\prime}}(0,T;W^{-1,p^{\prime}%
}(\Omega))
\]
Let us next check that $u=\widetilde{u}$. To that end, we observe that
\begin{align*}
\left\Vert u_{\ell}-\widetilde{u}_{\ell}\right\Vert _{L^{2}(Q)}^{2} &
=\sum_{m=0}^{\ell-1}\left(  \int_{mh}^{(m+1)h}(1-\tau)^{2}dt\right)
\left\Vert u_{m+1}-u_{m}\right\Vert _{L^{2}(\Omega)}^{2}\\
&  =\sum_{m=0}^{\ell-1}\left\Vert u_{m+1}-u_{m}\right\Vert _{L^{2}(\Omega
)}^{2}\int_{mh}^{(m+1)h}\left(  1+m-\frac{t}{h}\right)  ^{2}dt\\
&  =\frac{h}{3}\sum_{m=0}^{\ell-1}\left\Vert u_{m+1}-u_{m}\right\Vert
_{L^{2}(\Omega)}^{2}\\
&  \leq\frac{Ch}{3}%
\end{align*}
where for the last inequality above, we have used (\ref{3.3}) (in Lemma
\ref{l3.1}). We thus obtain, as $\ell\rightarrow\infty$, $u_{\ell}%
-\widetilde{u}_{\ell}\rightarrow0$ in $L^{2}(Q)$-strong. It follows from
(\ref{eq1920}) that
\[
\widetilde{u}_{\ell}=u_{\ell}+(\widetilde{u}_{\ell}-u_{\ell})\rightarrow
u\text{ in }L^{2}(Q)\text{-strong,}%
\]
so that $u=\widetilde{u}$. We therefore pass to the limit (as $\ell
\rightarrow\infty$) in (\ref{eq30}) and obtain
\[
\left\{
\begin{array}
[c]{l}%
{\displaystyle\int_{Q}}
(cu^{\prime}+w^{\prime})\varphi dxdt+%
{\displaystyle\int_{Q}}
\mathbf{v}\cdot\nabla\varphi dxdt=%
{\displaystyle\int_{Q}}
f\varphi dxdt\\
\text{for all }\varphi\in L^{p}(0,T;W_{0}^{1,p}(\Omega)).
\end{array}
\right.
\]
The next step is to identify the functions $\mathbf{v}$ and $w$ in terms of
$u$. Namely we must show that $\mathbf{v}=\boldsymbol{a}(\cdot,\nabla u)$ and
$w=\mathcal{W}(u)$. Let us first show that $\mathbf{v}=\boldsymbol{a}%
(\cdot,\nabla u)$. We proceed as classically to get, using the monotonicity of
$\boldsymbol{a}(x,\cdot)$ and the equality
\[
\limsup_{\ell\rightarrow\infty}\int_{Q}\boldsymbol{a}(\cdot,\nabla\tilde
{u}_{\ell})\cdot\nabla\tilde{u}_{\ell}dxdt=%
{\displaystyle\int_{Q}}
\mathbf{v}\cdot\nabla udxdt
\]
that
\[
\limsup_{\ell\rightarrow\infty}\int_{Q}\boldsymbol{a}(\cdot,\nabla\tilde
{u}_{\ell})\cdot\nabla(\tilde{u}_{\ell}-u)dxdt=\limsup_{\ell\rightarrow\infty
}\int_{Q}\boldsymbol{a}(\cdot,\nabla\tilde{u}_{\ell})\cdot\nabla\tilde
{u}_{\ell}dxdt-\int_{Q}\boldsymbol{a}(\cdot,\nabla\tilde{u}_{\ell})\cdot\nabla
udxdt\leq0.
\]
This implies that for all $\varphi\in L^{p}(0,T;W_{0}^{1,p}(\Omega))$,
\[
\liminf_{\ell\rightarrow\infty}\int_{Q}\boldsymbol{a}(\cdot,\nabla\tilde
{u}_{\ell})\cdot\nabla(\tilde{u}_{\ell}-\varphi)dxdt\geq\int_{Q}%
\boldsymbol{a}(\cdot,\nabla u)\cdot\nabla(u-\varphi)dxdt
\]
which implies
\begin{equation}
\liminf_{\ell\rightarrow\infty}\int_{Q}\boldsymbol{a}(\cdot,\nabla\tilde
{u}_{\ell})\cdot\nabla\tilde{u}_{\ell}dxdt-%
{\displaystyle\int_{Q}}
\mathbf{v}\cdot\nabla\varphi dxdt\geq\int_{Q}\boldsymbol{a}(\cdot,\nabla
u)\cdot\nabla(u-\varphi)dxdt.\label{6}%
\end{equation}
Choosing $\varphi=u$ in (\ref{6}) yields
\[
\liminf_{\ell\rightarrow\infty}\int_{Q}\boldsymbol{a}(\cdot,\nabla\tilde
{u}_{\ell})\cdot\nabla\tilde{u}_{\ell}dxdt\geq%
{\displaystyle\int_{Q}}
\mathbf{v}\cdot\nabla udxdt.
\]
It follows that
\[
\liminf_{\ell\rightarrow\infty}\int_{Q}\boldsymbol{a}(\cdot,\nabla\tilde
{u}_{\ell})\cdot\nabla\tilde{u}_{\ell}dxdt=%
{\displaystyle\int_{Q}}
\mathbf{v}\cdot\nabla udxdt,
\]
and thus
\[%
{\displaystyle\int_{Q}}
\mathbf{v}\cdot\nabla(u-\varphi)dxdt\geq\int_{Q}\boldsymbol{a}(\cdot,\nabla
u)\cdot\nabla(u-\varphi)dxdt\text{ for all }\varphi\in L^{p}(0,T;W_{0}%
^{1,p}(\Omega)).
\]
We deduce that $\mathbf{v}=\boldsymbol{a}(\cdot,\nabla u)$. Recalling that
$\tilde{f}_{\ell}\rightarrow f$ in $L^{2}(Q)$-strong, we get that $(u,w)$
satisfies the equation
\[
c\frac{\partial u}{\partial t}+\frac{\partial w}{\partial t}-{\Div}%
\boldsymbol{a}(\cdot,\nabla u)=f.
\]
It remains to check that the hysteresis equation in (\ref{eq1}) holds, that
is, $w=\mathcal{W}(u)$. We already remarked that the a priori estimates we
found yield
\begin{equation}
u_{\ell}\rightarrow u\text{ in }{L^{p}(0,T;W_{0}^{1,p}(\Omega))\cap
H^{1}(0,T;L^{2}(\Omega))}\text{-weak.}\label{1.11}%
\end{equation}
On the other hand, by interpolation and after a suitable choice of
representation in equivalence classes, we may deduce, from (\ref{eq9'}) (where
the last inclusion is also compact) that, possibly extracting a subsequence,
we have
\[
u_{\ell}\rightarrow u\text{ uniformly in }[0,T]\text{ and a.e. in }%
\Omega\text{.}%
\]
Using the strong continuity of the operator $\mathcal{W}$, we get that
\[
\mathcal{W}(u_{\ell};\cdot)\rightarrow\mathcal{W}(u;\cdot)\text{ uniformly in
}[0,T]\text{ and a.e. in }\Omega\text{.}%
\]
Now, we define the functions
\[
z_{\ell}(x,t)=\mathcal{W}[u_{\ell}(x,\cdot);x](t)\text{ (}\ell\in
\mathbb{N}\text{) and }z(x,t)=\mathcal{W}[u(x,\cdot);x](t).
\]
The compactness of the imbedding (\ref{eq8'}) yields that $u_{\ell}\rightarrow
u$ in $L^{2}(\Omega;\mathcal{C}\left(  [0,T]\right)  )$-strong; in particular,
$u_{\ell}(x,\cdot)\rightarrow u(x,\cdot)$ in $\mathcal{C}\left(  [0,T]\right)
$, for a.e $x\in\Omega$. The fact that $w=\mathcal{W}(u;\cdot)$ can be showed
arguing as in \cite[Section IV.1]{A.Visintin} in particular we have to use
some interpolation results and exploit the continuity of the hysteresis
operator $\mathcal{W}$ uniformly in time, a.e. in space, which can be deduced
from the locally Lipschitz continuity property of $\mathcal{W}$. Thus, using
the continuity of $\mathcal{W}$ assumed in (\textbf{A4}), we have $z_{\ell
}(x,\cdot)\rightarrow z(x,\cdot)$ in $\mathcal{C}\left(  [0,T]\right)  $, for
a.e. $x\in\Omega$. Next, note that, owing to the definition of the function
$h_{v,r}$ given in \cite{Kopfova},
\[
\sup_{0\leq t\leq T}\left\vert z_{\ell}(x,t)\right\vert \leq\kappa
_{0}(x)+\gamma_{0}\sup_{0\leq t\leq T}\left\vert u_{\ell}(x,t)\right\vert
\text{ for a.e. }x\in\Omega,
\]
where the right-hand side converges in $L^{2}(\Omega)$. Hence $z_{\ell
}\rightarrow z$ in $L^{2}(\Omega;\mathcal{C}\left(  [0,T]\right)  )$-strong.
Since $w_{\ell}$ is the linear interpolate of $z_{\ell}$, an analogous
argument shows that $w_{\ell}-z_{\ell}\rightarrow0$ in $L^{2}(\Omega
;\mathcal{C}\left(  [0,T]\right)  )$-strong.

In summary, as $w_{\ell}(x,\cdot)$ is the time interpolate given by
(\ref{eq27}), we have
\[
w_{\ell}\rightarrow\mathcal{W}(u;\cdot)\text{ uniformly in }[0,T]\text{ and
a.e. in }\Omega\text{.}%
\]
Therefore, by (\ref{eq37}) we get $w=\mathcal{W}(u;\cdot)$ a.e. in $Q$. By
(\ref{eq1918}), the sequence $(\left\Vert w_{\ell}(\cdot,t)\right\Vert
_{\mathcal{C}\left(  [0,T]\right)  })_{\ell}$ is uniformly integrable in
$\Omega$ as the same holds for $u_{\ell}$. Hence we have shown that $w_{\ell
}\rightarrow w=z$ in $L^{2}(\Omega;\mathcal{C}\left(  [0,T]\right)  )$-strong.
Finally we get that
\[
w\in L^{2}(Q)\cap L^{2}(\Omega;\mathcal{C}([0,T]))\text{ with }w^{\prime}\in
L^{p^{\prime}}(0,T;W^{-1,p^{\prime}}(\Omega)).
\]
This concludes the proof of the existence issue.

\section{Proof of Theorem \ref{t4.1}: Uniqueness result\label{subsec3.2}}

The main purpose of this section is to prove uniqueness of the solution to
(\ref{eq1}) together with the estimate (\ref{1.9}). It is important to remark
that no information concerning the uniqueness of the solution is presented in
the preceding subsection. This question has indeed remained unanswered for a
number of years, and it is Hilpert \cite{Hilpert} who finally developed a
technique to shown that in quite general situation, the solution of the
initial-boundary value problem (\ref{eq1}) does in fact continuously depends
on the right-hand side $f$ and on the initial data. His method will be
presented in the sequel.

The next results can be found in \cite{Hilpert} or in \cite{Brokate} in which
some slightly modified results have been stated and proved, but we recall them
here for the convenience of the reader. The following  inequality will play a
key role.

\begin{proposition}
[Hilpert's Inequality]\label{p3.1}Consider the hysteresis operator
$\mathcal{W}$ given by
\begin{equation}
\mathcal{W}\left[  v;w_{-1}\right]  (t)=q\left(  \mathcal{F}\left[
v,w_{-1}\right]  (t)\right)  ,\quad0\leq t\leq T, \label{eq40}%
\end{equation}
with $w_{-1}\in\mathbb{R}$, where $q\in W_{loc}^{1,\infty}(\mathbb{R})$ is an
increasing function and where $\mathcal{F}$ is a hysteresis operator. Suppose
that $v_{1},v_{2}\in W^{1,1}(0,T)$ and $w_{-1,1},\ w_{-1,2}\in\mathbb{R}$ are
given, and let $v=v_{2}-v_{1}$, $w=w_{2}-w_{1}$, where $w_{i}=\mathcal{W}%
[v_{i},w_{-1,i}]$,$\ i=1,2$. Then
\begin{equation}
\frac{d}{dt}w_{+}(t)\leq w^{\prime}(t)H\left(  v(t)\right)  \text{ a.e. in
}\left(  0,T\right)  , \label{eq41}%
\end{equation}

where $w_{+}=\max{\left\{  w,0\right\}  }$ and where $H$ denotes the Heaviside function.
\end{proposition}

\begin{proof}
$\mathcal{F}$ maps $W^{1,1}(0,T)$ into itself. Hence, the chain rule can be
applied to (\ref{eq40}), and the time derivatives in (\ref{eq41}) are defined
almost everywhere. If $q(x)=x$, i.e. if $\mathcal{W}[\cdot;w_{-1}%
]=\mathcal{F}[\cdot;w_{-1}]$, that the crucial implication%
\begin{equation}
w_{2}(t)<w_{1}(t),\ v_{2}(t)\geq v_{1}(t)\Longrightarrow w_{2}^{\prime}%
(t)\geq0\text{, }w_{1}^{\prime}(t)\leq0, \label{eq42}%
\end{equation}
holds almost everywhere in $(0,T)$. Since (\ref{eq42}) remains for almost
every $t\in(0,T)$ if $w_{i}(t)$ is replace by $q\left(  w_{i}(t)\right)  $, we
see that (\ref{eq42}) is true for any increasing $q\in W_{loc}^{1,\infty
}(\mathbb{R})$.

Now, (\ref{eq42}) implies that%
\[
0\leq w^{\prime}(t)H(v(t))\text{ if }w_{2}(t)<w_{1}(t).
\]
Interchanging the indices $1$ and $2$, we also get
\[
w^{\prime}(t)\leq w^{\prime}(t)H(v(t))\text{ if }w_{1}(t)<w_{2}(t).
\]
Finally, on the set $\left\{  t:w_{1}(t)=w_{2}(t)\right\}  $ both sides of
(\ref{eq41}) vanish, which concludes the proof of the assertion.
\end{proof}

We now present the general stability result. In addition to (\textbf{A4}), we
need further Assumption (\textbf{A5}) on the hysteresis operators
$\mathcal{W}[\cdot;x]$, $x\in\Omega$.

\begin{remark}
\label{r4.1}\emph{Let (\textbf{A5}) be satisfied. Since any weak solution
}$(u,w)$\emph{ in the sense of Theorem \ref{t4.1} satisfies}%
\[
u\in H^{1}\left(  0,T;L^{2}(\Omega)\right)  \ \left(  =L^{2}\left(
\Omega;H^{1}(0,T)\right)  \subset L^{2}\left(  \Omega;W^{1,1}(0,T)\right)
\right)  ,
\]
\emph{we can conclude that }$w\in L^{2}\left(  \Omega;H^{1}(0,T)\right)
$\emph{.}
\end{remark}

\begin{theorem}
[$L^{1}$-Stability for the nonlinear heat equation with hysteresis]%
\label{t4.2}Let \emph{(\textbf{A4})} hold, and let $u_{0,1}$, $u_{0,2}\in
W_{0}^{1,p}(\Omega)$ and $f_{1},f_{2}\in L^{2}(Q)$ be given. Suppose that the
parameterized hysteresis operator $\mathcal{W}[\cdot;x]$ satisfies
\emph{(\textbf{A5})}, and, at every space point $x\in\Omega$, the inequality
\emph{(\ref{eq41})}. Then any pair $(u_{1},w_{1})$ and $(u_{2},w_{2})$ of weak
solutions to \emph{(\ref{eq1})} in the sense of Theorem \emph{\ref{t4.1}}
satisfies, for almost every $t\in(0,T)$,
\begin{equation}%
\begin{array}
[c]{l}%
{\displaystyle\int_{\Omega}}
{c\left\vert u_{2}-u_{1}\right\vert (x,t)dx}+%
{\displaystyle\int_{\Omega}}
{\left\vert w_{2}-w_{1}\right\vert (x,t)dx}\leq%
{\displaystyle\int_{\Omega}}
{\left\vert u_{0,2}-u_{0,1}\right\vert (x)dx}\\
\ \ \ \ \ \ \ +%
{\displaystyle\int_{\Omega}}
{\left\vert w_{2}-w_{1}\right\vert (x,0)dx}+%
{\displaystyle\int_{0}^{t}}
{%
{\displaystyle\int_{\Omega}}
{\left\vert f_{2}-f_{1}\right\vert (x,\tau)dx}d\tau}.
\end{array}
\label{eq48}%
\end{equation}

\end{theorem}

\begin{proof}
We proceed as in \cite{Brokate}. Let $H_{\varepsilon}:\mathbb{R}%
\rightarrow\mathbb{R}$ denote the regularized Heaviside function defined by
\[
H_{\varepsilon}(x)=\left\{
\begin{array}
[c]{l}%
1\text{ if }x\geq\varepsilon\\
\frac{x}{\varepsilon}\text{\ if }0\leq x\leq\varepsilon\\
0\text{ if }x\leq0.
\end{array}
\right.
\]
We set $u=u_{2}-u_{1}$ and $w=w_{2}-w_{1}$, where $w_{i}(x,\cdot
)=\mathcal{W}[u_{i}(x,\cdot);x]$, $i=1,2$. Clearly $H_{\varepsilon}\circ u\in
L^{\infty}(Q)\cap L^{\infty}\left(  0,T;W_{0}^{1,p}(\Omega)\right)  $, since
$H_{\varepsilon}$ is Lipschitz continuous. Hence, we may test the difference
of the variational equations (\ref{eq14}) for the pairs $(u_{i},w_{i})$
($i=1,2$) and the right-hand sides $f_{i}$ by the function $\varphi=\left(
H_{\varepsilon}\circ u\right)  1_{(0,t)}$, for $t\in(0,T)$, to obtain
\begin{equation}
\left\{
\begin{array}
[c]{l}%
{\displaystyle\int_{0}^{t}}
{\displaystyle\int_{\Omega}}
cu^{\prime}H_{\varepsilon}(u)dxd\tau+%
{\displaystyle\int_{0}^{t}}
\left\langle {{w^{\prime},H_{\varepsilon}(u)}}\right\rangle d\tau\\
\ +%
{\displaystyle\int_{0}^{t}}
{\displaystyle\int_{\Omega}}
(\boldsymbol{a}(\cdot,\nabla u_{2})-\boldsymbol{a}(\cdot,\nabla u_{1}%
))\cdot\nabla\left(  H_{\varepsilon}\circ u\right)  dxd\tau=%
{\displaystyle\int_{0}^{t}}
{\displaystyle\int_{\Omega}}
\left(  f_{2}-f_{1}\right)  H_{\varepsilon}(u)dxd\tau.
\end{array}
\right.  \label{eq50}%
\end{equation}
Applying the chain rule to the third integrand of the left-hand side of
(\ref{eq50}), we obtain
\[
(\boldsymbol{a}(\cdot,\nabla u_{2})-\boldsymbol{a}(\cdot,\nabla u_{1}%
))\cdot\nabla\left(  H_{\varepsilon}\circ u\right)  =H_{\varepsilon}^{\prime
}(u)(\boldsymbol{a}(\cdot,\nabla u_{2})-\boldsymbol{a}(\cdot,\nabla
u_{1}))\cdot\nabla(u_{2}-u_{1})\geq0,
\]
the last inequality being a consequence of the monotonicity of $\boldsymbol{a}%
(x,\cdot)$. Hence, in view of (\ref{eq50})
\[
\int_{0}^{t}{\int_{\Omega}{\left(  cu^{\prime}+w^{\prime}\right)
(x,\tau)H_{\varepsilon}\left(  u(x,\tau)\right)  dx}d\tau}\leq\int_{0}%
^{t}{\int_{\Omega}{\left(  f_{2}-f_{1}\right)  (x,\tau)H_{\varepsilon}\left(
u(x,\tau)\right)  dx}d\tau}.
\]
Since $u^{\prime}$, $w^{\prime}\in L^{1}(Q)$ and $\left\vert H_{\varepsilon
}\circ u\right\vert \leq1$, we may pass to the limit as $\varepsilon
\rightarrow0$ to arrive at
\begin{equation}%
\begin{array}
[c]{l}%
{\displaystyle\int_{\Omega}}
{cu_{+}(x,t)dx}+%
{\displaystyle\int_{\Omega}}
{%
{\displaystyle\int_{0}^{t}}
{w^{\prime}(x,\tau)H\left(  u(x,\tau)\right)  d\tau}dx}\leq\\
\ \ \ \ \
{\displaystyle\int_{\Omega}}
{cu_{+}(x,0)dx}+%
{\displaystyle\int_{0}^{t}}
{%
{\displaystyle\int_{\Omega}}
{\left(  f_{2}-f_{1}\right)  (x,\tau)H\left(  u(x,\tau)\right)  dx}d\tau}.
\end{array}
\label{eq53}%
\end{equation}
We estimate the second integral on the left side of (\ref{eq53}) from below
using (\ref{eq41}) to obtain
\begin{equation}%
\begin{array}
[c]{l}%
{\displaystyle\int_{\Omega}}
{cu_{+}(x,t)dx}+%
{\displaystyle\int_{\Omega}}
{w_{+}(x,t)dx}\leq%
{\displaystyle\int_{\Omega}}
{cu_{+}(x,0)dx}\\
\ \ +%
{\displaystyle\int_{\Omega}}
{w_{+}(x,0)dx}+%
{\displaystyle\int_{0}^{t}}
{%
{\displaystyle\int_{\Omega}}
{\left(  f_{2}-f_{1}\right)  (x,\tau)H\left(  u(x,\tau)\right)  dx}d\tau}.
\end{array}
\label{eq54}%
\end{equation}
We now reverse the role of the indices $1$ and $2$ and add (\ref{eq54}) to the
corresponding inequality. Since $0\leq H(v)+H(-v)\leq1$, the resulting
inequality yields (\ref{eq48}).
\end{proof}

\begin{corollary}
[Uniqueness]\label{c.41}Under the assumptions of Theorem \emph{\ref{t4.2}},
the weak solution in the sense of Theorem \emph{\ref{t4.1}} is unique.
\end{corollary}

\begin{proof}
Assuming $f_{1}=f_{2}$ and $u_{0,1}=u_{0,2}$ in (\ref{eq48}) yield at once
$u_{1}=u_{2}$. This ends the proof of Theorem \ref{t4.1}.
\end{proof}

\begin{proof}
[Proof of estimate \emph{(\ref{1.9})}]We first note that thanks to Remark
\ref{r3.1}, the equality (\ref{1.1}) (in Lemma \ref{l1.1}) holds true. Indeed,
we may rewrite the leading equation in (\ref{eq1}) under the form
\begin{equation}
{\Div}\boldsymbol{a}(\cdot,\nabla u)=-f+cu^{\prime}+w^{\prime}.\label{1.12}%
\end{equation}
Using assumption (\textbf{A5}) we obtain that $w^{\prime}\in L^{2}%
(0,T;L^{2}(\Omega))$. According to (\textbf{A2}), $c\in L^{\infty}(\Omega)$
and since it is also known that $u^{\prime}\in L^{2}(0,T;L^{2}(\Omega))$ (see
e.g., (\ref{1.11})) then $cu^{\prime}\in L^{2}(0,T;L^{2}(\Omega))$. So that,
because of assumption (\textbf{A3}) on $f$, the equality (\ref{1.12}) yields
${\Div}\boldsymbol{a}(\cdot,\nabla u)\in L^{2}(0,T;L^{2}(\Omega))$. The
assumptions given in Remark \ref{r3.1} are thus satisfied, in such a way that
(\ref{1.1}) holds, that is,
\[
\frac{d}{dt}\sigma(u(t))=-\left(  {\Div}\boldsymbol{a}(\cdot,\nabla
u(t)),u^{\prime}(t)\right)  \text{ a.e. }t\in\lbrack0,T].
\]
Therefore we multiply (\ref{eq1}) by $u^{\prime}(t)$ and integrate over
$\Omega\times\lbrack t_{1},t_{2}]$ where $0\leq t_{1}<t_{2}\leq T$, and
proceed as we did in obtaining (\ref{eq1.7}). This yields at once (\ref{1.9}).
\end{proof}

\section{Long time behaviour: Proof of Theorem \ref{t1.2}}

We are concerned here with the proof of Theorem \ref{t1.2}.

\begin{proof}
[Proof of Theorem \emph{\ref{t1.2}}]Let $(t_{n})_{n}$ be a sequence of times
satisfying $0\leq t_{n}\rightarrow\infty$ as $n\rightarrow\infty$. Let $u$ be
determined by Theorem \ref{t4.1} (we do not need uniqueness at this level).
Then owing to (\ref{eq1.7}), we have that $u\in\mathcal{C}([0,\infty
);L^{2}(\Omega))$, so that $u(\cdot,t_{n})\equiv u(t_{n})$ makes sense for all
$n$, and we have $\sup_{n}\left\Vert \nabla u(t_{n})\right\Vert _{L^{p}%
(\Omega)}\leq C$ where $C>0$ is independent of $n$. Therefore, up to a
subsequence of $(t_{n})_{n}$ not relabeled, there exists a function
$u_{\infty}\in W_{0}^{1,p}(\Omega)$ such that $u(t_{n})\rightarrow u_{\infty}$
in $W_{0}^{1,p}(\Omega)$-weak and in $L^{p}(\Omega)$-strong. This shows
(\ref{1.4}).

The next step is to check that $u_{\infty}$ solves (\ref{1.5}). To proceed
with, let us first observe that assuming $f$ depending on the time variable
$t$, the hypothesis (\textbf{A3})$_{1}$ yields that $f\in\mathcal{C}%
([0,\infty);L^{2}(\Omega))\cap L^{2}(0,\infty;L^{2}(\Omega))$, so that
$f(t_{n})\rightarrow0$ in $L^{2}(\Omega)$-strong as $t\rightarrow+\infty$. Of
course, if $f$ does not depend on $t$, then we do not need any further
requirement on $f$ (like (\textbf{A3})$_{1}$), but only (\textbf{A3}). This
being so, let $u_{n}(t)=u(t+t_{n})$ for $t\in\lbrack0,1]$, and $w_{n}%
=\mathcal{W}[u_{n};\cdot]$. Then it is a fact that $u_{n}\in L^{\infty
}(0,1;W_{0}^{1,p}(\Omega))\cap\mathcal{C}([0,1];L^{2}(\Omega))$ solves the
equation
\begin{equation}
\left\{
\begin{array}
[c]{l}%
\frac{\partial}{\partial t}\left(  cu_{n}+w_{n}\right)  -{\Div}\boldsymbol{a}%
(\cdot,\nabla u_{n})=f_{n}\text{ in }Q_{1}=\Omega\times(0,1)\\
w_{n}(x)=\mathcal{W}[u_{n}(x,\cdot);x]\text{ in }\Omega\\
u_{n}=0\text{ on }\partial\Omega\times(0,1)\text{ and }u_{n}(0)=u(t_{n})\text{
in }\Omega
\end{array}
\right.  \label{1.6}%
\end{equation}
where $f_{n}(t)=f(t+t_{n})$ for $t\in\lbrack0,1]$. Proceeding as in Subsection
\ref{subsec3.1.2} we obtain the estimate
\begin{equation}
\left\Vert u_{n}^{\prime}\right\Vert _{L^{2}(Q_{1})}+\sup_{0\leq t\leq
1}\left\Vert \nabla u_{n}\right\Vert _{L^{p}(\Omega)}\leq C \label{1.7}%
\end{equation}
where $C>0$ does not depend neither on $t$, nor on $n$. Next, following the
lines of Subsection \ref{subsec3.1.3}, we infer the \ existence of $v\in
L^{\infty}(0,1;W_{0}^{1,p}(\Omega))\cap\mathcal{C}([0,1];L^{2}(\Omega))$ with
$v^{\prime}\in L^{2}(0,1;L^{2}(\Omega))$ such that, up to a subsequence of
$(u_{n})_{n}$ keeping the same notation,
\[
u_{n}\rightarrow v\text{ in }L^{p}(0,1;W_{0}^{1,p}(\Omega))\text{-weak and in
}L^{\infty}(0,1;W_{0}^{1,p}(\Omega))\text{-weak}\ast
\]%
\[
u_{n}^{\prime}\rightarrow v^{\prime}\text{ in }L^{2}(Q_{1})\text{-weak}%
\]
and
\[
\boldsymbol{a}(\cdot,\nabla u_{n})\rightarrow\boldsymbol{a}(\cdot,\nabla
v)\text{ in }L^{p^{\prime}}(Q_{1})\text{-weak.}%
\]
Moreover using (\ref{1.7}) (or (\ref{eq1.7})) we see that
\[
u_{n}^{\prime}=u^{\prime}(\cdot+t_{n})\rightarrow0\text{ in }L^{2}%
(0,1;L^{2}(\Omega))\text{-strong.}%
\]
It follows that $v^{\prime}=0$, in such a way that $v$ is constant with
respect to $t$, that is, $v(t)=v(0)$ in $L^{2}(\Omega)$ for all $t\in
\lbrack0,1]$. However it emerges from the equality $u_{n}(0)=u(t_{n})$ (which
yields $v(0)=u_{\infty}$) that $v(t)=u_{\infty}$ for all $t\in\lbrack0,1]$.
Therefore we obtain $w_{n}\rightarrow\mathcal{W}[u_{\infty};\cdot]$ in
$L^{2}(0,1;L^{2}(\Omega))$-weak, and $\mathcal{W}[u_{\infty};\cdot]$ does not
depend on $t$. It readily follows that $v\equiv u_{\infty}\in W_{0}%
^{1,p}(\Omega)$ solves the equation (\ref{eq14}), which amounts to (\ref{1.5})
by suitable choice of test functions (namely choose $\varphi$ under the form
$\varphi(x,t)=\chi(t)\phi(x)$ with $\chi\in C_{0}^{\infty}(0,1)$ and $\phi\in
W_{0}^{1,p}(\Omega)$). This concludes the proof of the theorem.
\end{proof}

\textbf{Conflict of interests.}\\
The authors declare that there is no conflict of interest regarding the publication
of this paper.


\begin{thebibliography}{99}                                                                                               %
\addcontentsline{toc}{chapter}{Bibliograhie}

\bibitem {Adams}A.S. Adams, Sobolev Spaces, Acad. Press, New York, 1975.

\bibitem {Brezis}H. Br\'{e}zis, Analyse Fonctionnelle: Th\'{e}orie et
Applications, Masson, Paris, 1983.

\bibitem {Brezis2}H. Br\'{e}zis, Op\'{e}rateurs Maximaux Monotones et
Semi-groups de Contractions dans les Espaces de Hilbert, North-Holland,
Amsterdam, 1973.

\bibitem {Brokate}M. Brokate, J. Sprekels, Hysteresis and Phase Transitions,
Appl. Math. Sci., Vol. \textbf{121}, Springer, New York, 1996.

\bibitem {Colli}P. Colli, A. Visintin, on a class of doubly nonlinear
evolutions equations, Comm. PDE \textbf{15} (1990) 737--756.

\bibitem {Chill}R. Chill, M.A. Jendoubi, Convergence to steady states in
asymptotically autonomous semilinear evolution equations, Nonlin. Anal.
\textbf{53} (2003) 1017--1039.

\bibitem {Eleuteri}M. Eleuteri, On some P.D.E.s with hysteresis, Ph.D thesis,
UTM PhDTS no. \textbf{47}, Trento, 2006.

\bibitem {Eleuteri2}M. Eleuteri, P. Krejci, Asymptotic behaviour of a Neumann
parabolic problem with hysteresis, ZAMM- Z. Angew. Math. Mech. \textbf{87}
(2007) 261--277.

\bibitem {Ewing}J.W. Ewing, Experimental researches in magnetism, Trans. R.
Soc. Lond. \textbf{176} (1885) 523--640.

\bibitem {Flynn}D. Flynn, Application of the Preisach model in soil-moisture
hysteresis, Master thesis, University College Cork, 2004.

\bibitem {Francu}J. Francu, P. Krejci, Homogenization of scalar wave equations
with hysteresis, Continuum Mech. Thermodyn. \textbf{11} (1999) 371--390.

\bibitem {Hilpert}M. Hilpert, On uniqueness for evolution problems with
hysteresis. In Mathematical Models for Phase Change Problems (J.F. Rodrigues,
ed.). Birkh\"{a}user, Basel (1989) 377--388.

\bibitem {Pokrovskii}M.A. Krasnosel'skii, A.V. Pokrovskii, Systems with
Hysteresis, Nauka, Moscow, 1983 (English edition, Springer, Berlin, 1989).

\bibitem {KV1994}N. Kenmochi, A. Visintin, Asymptotic stability for nonlinear
PDEs with hysteresis, Euro. J. Appl. Math. \textbf{5} (1994) 39--56.

\bibitem {Kopfova}J. Kopfov\'{a}, A convergence result for spatially
inhomogeneous Preisach operators, ZAMP \textbf{58} (2007) 350--356.

\bibitem {Krejci}P. Krejci, Hysteresis, Convexity and dissipation in
hyperbolic equations, Gakuto Int. Series Math. Sci. \& Appl., Vol. \textbf{8},
Gakkotosho, Tokyo, 1996.

\bibitem {Lions}J.L. Lions, Quelques m\'{e}thodes de r\'{e}solution des
probl\`{e}mes aux limites non lin\'{e}aires, Dunod, Paris, 1969.

\bibitem {Marcellini}P. Marcellini, Approximation of quasiconvex functions,
and lower semicontinuity of multiple integrals, Manuscripta Math. \textbf{51}
(1985) 1--28.

\bibitem {Mayergoyz}I.D. Mayergoyz, Mathematical Models of Hysteresis,
Springer-Verlag, New York, 1991.

\bibitem {Mielke}A. Mielke, S. Zelik, On the vanishing-viscosity limit in
parabolic systems with rate-independent dissipations terms, Ann. Sc. Norm.
Super. Pisa. CI. Sci. \textbf{13} (2014) 67--135.

\bibitem {Stefanelli}G. Schimperna, A. Segatti, U. Stefanelli, Well-posedness
and long-time behavior for a class of doubly nonlinear equations, Discrete
Cont. Dyn. Syst. -A \textbf{18} (2007) 15--38.

\bibitem {Lazlo}L. Simon, Application of monotone type operators to nonlinear
PDEs, Prime Rate Kft., Budapest, 2013.

\bibitem {Tao}T. Tao, An introduction to measure theory, Graduate studies in
Math., Vol.\textbf{126}, AMS, Providence, RI, 2011.

\bibitem {Visintin}A. Visintin, On the Preisach model for hysteresis,
Nonlinear Analysis T.M.A. \textbf{9} (1984) 977--996.

\bibitem {A.Visintin}A. Visintin, Differential Models of Hysteresis, Springer,
Berlin, 1994.
\end{thebibliography}
\end{document}